\documentclass[reqno]{gtpart}
\usepackage{hyperref}
\usepackage{stmaryrd}

\newif\ifshowkeys
\showkeysfalse

\ifshowkeys
\newcommand{\lbl}[1]{\label{#1}\textup{[\texttt{#1}]}\par}
\else
\newcommand{\lbl}{\label}
\fi

\usepackage[matrix,arrow]{xy}
\newdir{ >}{{}*!/-9pt/\dir{>}}

\newcommand{\Hom}       {\operatorname{Hom}}

\newcommand{\head}	{\operatorname{head}}
\newcommand{\tail}	{\operatorname{tail}}

\newcommand{\CA}        {{\mathcal{A}}}
\newcommand{\CB}        {{\mathcal{B}}}
\newcommand{\CC}        {{\mathcal{C}}}
\newcommand{\CH}        {{\mathcal{H}}}
\newcommand{\CI}        {{\mathcal{I}}}

\newcommand{\CP}        {{\mathcal{P}}}
\newcommand{\CL}        {{\mathcal{L}}}
\newcommand{\CR}        {{\mathcal{R}}}
\newcommand{\CRb}       {{\mathcal{R}}_{\text{bool}}}
\newcommand{\CRl}       {{\mathcal{R}}_{\text{latt}}}
\newcommand{\CS}        {{\mathcal{S}}}
\newcommand{\bCL}	{\overline{\mathcal{L}}}
\newcommand{\hCL}	{\widehat{\mathcal{L}}}

\newcommand{\N}         {{\mathbb{N}}}
\newcommand{\Ni}        {{\mathbb{N}_{\infty}}}
\newcommand{\Nw}        {{\mathbb{N}_{\omega}}}

\newcommand{\Zpl}       {{\mathbb{Z}_{(p)}}}      
\newcommand{\QZpl}      {{\mathbb{Q}/\mathbb{Z}_{(p)}}}      

\newcommand{\Om}        {\Omega}
\newcommand{\Sg}        {\Sigma}

\newcommand{\al}        {\alpha}
\newcommand{\ep}        {\epsilon}
\newcommand{\tht}       {\theta}
\newcommand{\sg}        {\sigma}
\newcommand{\om}        {\omega}
\newcommand{\bsg}	{\overline{\sigma}}
\newcommand{\bphi}	{\overline{\phi}}

\newcommand{\bc}[1]	{\langle #1\rangle}
\newcommand{\ip}[1]	{\langle #1\rangle}

\newcommand{\ot}        {\otimes}
\newcommand{\tj}	{\widetilde{\jmath}}
\newcommand{\sm}        {\setminus}
\newcommand{\sse}       {\subseteq}
\newcommand{\st}        {\;|\;}
\newcommand{\xra}       {\xrightarrow}
\newcommand{\ov}[1]     {\overline{#1}}

\newcommand{\Smash}     {\wedge}
\newcommand{\Wedge}     {\vee}
\newcommand{\bigWedge}  {\bigvee}

\newcommand{\convto}    {\Longrightarrow}
\newcommand{\colim}  {\operatornamewithlimits{\underset{\longrightarrow}{lim}}}

\renewcommand{\:}{\colon}

\newtheorem{theorem}{Theorem}[section]

\newtheorem{lemma}[theorem]{Lemma}
\newtheorem{proposition}[theorem]{Proposition}
\newtheorem{corollary}[theorem]{Corollary}
\theoremstyle{definition}
\newtheorem{remark}[theorem]{Remark}
\newtheorem{definition}[theorem]{Definition}

\newwrite\refs
\openout\refs=\jobname.refs
\makeatletter
\renewcommand\@setref[3]{%
        \ifx#1\relax
                \write\refs{'#3' \thepage\space undefined}%
                \protect \G@refundefinedtrue
                \nfss@text{\reset@font\bfseries ??}%
                \@latex@warning{Reference `#3' on page \thepage\space
                                undefined}%
        \else
                \write\refs{'#3' \thepage\space
                            \expandafter\@secondoftwo#1}%
                \expandafter#2#1\null
        \fi
}
\makeatother

\begin{document}
\title{A combinatorial model for the known Bousfield classes}
\author{N.~P.~Strickland}
\address{
School of Mathematics and Statistics, 
The University of Sheffield,
Sheffield S3 7RH, 
UK
}
\email{N.P.Strickland@sheffield.ac.uk}

\begin{abstract}
 We give a combinatorial construction of an ordered semiring $\CA$,
 and show that it can be identified with a certain subquotient of the
 semiring of $p$-local Bousfield classes, containing almost all of the
 classes that have previously been named and studied.  This is a
 convenient way to encapsulate most of the known results about
 Bousfield classes.
\end{abstract}

\maketitle 

\section{Introduction}

Fix a prime $p$, and let $\CL$ denote the set of Bousfield classes
in the $p$-local stable category (which can be regarded as an ordered
semiring in a natural way).  This note is an attempt to organise
many of the known results about the structure of $\CL$ in a more
coherent way.

One of the main open questions about $\CL$ is Ravenel's Telescope
Conjecture (TC).  The statement will be recalled in
Remark~\ref{rem-tc}.  Many people suspect that TC is false, but this
has still not been proven.  We will define an ordered semiring $\bCL$
which is, in a certain sense, the largest quotient of $\CL$ in which
TC becomes true.  We will then define (in an explicit, combinatorial
way) another ordered semiring $\CA$ and a function $\phi\:\CA\to\CL$
such that the composite
\[ \CA \xra{\phi} \CL \xra{\pi} \bCL \]
is an injective homomorphism of ordered semirings.  (However, $\phi$
itself is probably not a semiring homomorphism, unless TC holds.) For
almost all elements $x\in\CL$ that have been named and studied, we
have $\pi(x)\in\pi\phi(\CA)$.  Thus, $\CA$ is a good model for the
known part of $\CL$.

\begin{remark}\lbl{rem-exceptions}
 We will mention two exceptions to the idea that $\CA$ captures all
 known phenomena in $\bCL$.  First consider the spectra $BP/J$
 from~\cite[Definition 2.7]{ra:lrc}, where $J$ is generated by an
 invariant regular sequence of infinite length.  Ravenel shows that
 for different $J$ these have many different Bousfield classes, but
 only one of them is in the image of $\phi\:\CA\to\CL$.  It is
 unlikely that the situation is any better in $\bCL$.  Next, the
 paper~\cite{marash:tls} introduces a number of new Bousfield classes
 in the course of studying the telescope conjecture.  It is reasonable
 to conjecture that their images in $\bCL$ lie in $\pi\phi(\CA)$, but
 we have not considered this question carefully.
\end{remark}

\section{Ordered semirings}

\begin{definition}\lbl{defn-osr}
 By an \emph{ordered semiring} we will mean a set $\CR$ equipped with
 elements $0,1\in\CR$ and binary operations $\Wedge$ and $\Smash$ such
 that:
 \begin{itemize}
  \item[(a)] $\Wedge$ is commutative and associative, with $0$ as an
   identity element.
  \item[(b)] $\Smash$ is commutative and associative, with $1$ as an
   identity element.
  \item[(c)] $\Smash$ distributes over $\Wedge$.
  \item[(d)] For all $u\in\CR$ we have $0\Smash u=0$ and $1\Wedge u=1$
   and $u\Wedge u=u$.
 \end{itemize}
\end{definition}
It is easy to check that there is a natural partial order on such an
object, where $u\leq v$ iff $u\Wedge v=v$.  The binary operations
preserve this order, and $0$ and $1$ are the smallest and largest
elements.  Moreover, $u\Wedge v$ is the smallest element satisfying
$w\geq u$ and $w\geq v$.  

\begin{definition}
 Let $\CR$ be an ordered semiring.  We say that $\CR$ is
 \emph{complete} if every subset $S\sse\CR$ has a least upper bound 
 $\bigWedge S\in\CR$.  We say that $\CR$ is \emph{completely
  distributive} if, in addition, for all $S\sse\CR$ and $x\in\CR$ we
 have 
 \[ \bigWedge\{x\Smash s\st s\in S\} = x\Smash\bigWedge S. \]
\end{definition}

We next recall the definition of the Bousfield semiring $\CL$.
\begin{definition}\lbl{defn-L}
 We write $\CB$ for the category of $p$-local spectra in the sense of
 stable homotopy theory.  This has a coproduct, which is written
 $X\Wedge Y$ and is also called the wedge product.  There is also a
 smash product, written $X\Smash Y$.  Up to natural isomorphism, both
 operations are commutative and associative, and the smash product
 distributes over the wedge product.  The $p$-local sphere spectrum
 $S$ is a unit for the smash product, and the zero spectrum is a unit
 for the wedge product.  

 For any object $E\in\CB$ we put 
 \[ \bc{E} = \{X\in\CB\st E\Smash X=0\}, \]
 and call this the \emph{Bousfield class} of $E$.  We then put 
 \[ \CL = \{\bc{E}\st E\in\CB\}. \]
 (This is a set rather than a proper class, by a theorem of
 Ohkawa~\cite{oh:ihh,dwpa:ot}.)  It is straightforward to check
 that this has well-defined operations satisfying
 \begin{align*}
  \bc{E}\Wedge\bc{F} &= \bc{E\Wedge F} = \bc{E}\cap\bc{F} \\
  \bc{E}\Smash\bc{F} &= \bc{E\Smash F}
   = \{X\st E\Smash X\in\bc{F}\} = \{X\st F\Smash X\in\bc{E}\}.
 \end{align*}
 It is then easy to check that this gives an ordered semiring, with
 top and bottom elements as follows:
 \begin{align*}
  1 &= \bc{S} = \{0\} \\
  0 &= \bc{0} = \CB.
 \end{align*}
 The resulting ordering of $\CL$ is given by $\bc{E}\leq\bc{F}$ iff
 $\bc{E}\supseteq\bc{F}$.  

 Next, recall that $\CB$ has a coproduct (written $\bigWedge_iX_i$)
 for any family of objects $(X_i)_{i\in I}$, and these satisfy 
 $W\Smash\bigWedge_iX_i\simeq\bigWedge_i(W\Smash X_i)$.  It follows
 that $\CL$ is completely distributive, with
 $\bigWedge_i\bc{E_i}=\bc{\bigWedge_iE_i}$.  
\end{definition}

\begin{definition}\lbl{defn-quot}
 Let $\CR$ be an ordered semiring, and let $\ep\in\CR$ be an
 idempotent element (so $\ep\Smash\ep = \ep$).  We put 
 \[ \CR/\ep = \{a\in\CR\st a\geq\ep\}. \]
 We define a surjective function $\pi\:\CR\to\CR/\ep$ by
 $\pi(a)=a\Wedge\ep$.
\end{definition}

\begin{proposition}\lbl{prop-quot}
 There is a unique ordered semiring structure on $\CR/\ep$ such that
 $\pi$ is a homomorphism.  Moreover, if $\phi\:\CR\to\CS$ is any
 homomorphism of ordered semirings with $\phi(\ep)=0$, then there is a
 unique homomorphism $\ov{\phi}\:\CR/\ep\to\CS$ with
 $\ov{\phi}\circ\pi=\phi$. 
\end{proposition}
\begin{proof}
 The set $\CR/\ep$ clearly contains $1$ and is closed under $\Smash$
 and $\Wedge$.  We claim that these operations make $\CR/\ep$ into an
 ordered semiring, with $\ep$ as a zero element.  All axioms not
 involving zero are the same as the corresponding axioms for $\CR$.
 The axioms involving zero say that we should have $\ep\Wedge u=u$ and
 $\ep\Smash u=\ep$ for all $u\in\CR/\ep$, and this follows directly
 from the definition of $\CR/\ep$ and the idempotence of $\ep$.  It is
 clear that this is the unique structure on $\CR/\ep$ for which $\pi$
 is a homomorphism.  If $\phi\:\CR\to\CS$ has $\phi(\ep)=0$ then we
 can just take $\ov{\phi}$ to be the restriction of $\phi$ to $\CR/\ep$.
 This is clearly a homomorphism, with 
 \[ \ov{\phi}(\pi(a)) = \phi(a\Wedge\ep) = \phi(a) \Wedge \phi(\ep) =
     \phi(a)\Wedge 0 = \phi(a)
 \]
 as required.
\end{proof}

\begin{remark}\lbl{rem-quot-complete}
 If $\CR$ is complete, or completely distributive, then we find that
 $\CR/\ep$ has the same property.
\end{remark}

\begin{remark}\lbl{rem-defn-bCL}
 In Definition~\ref{defn-basic} we will introduce certain Bousfield
 classes $a(n)=\bc{C_nK'(n)}$ for $n\in\N$, and put
 $\ep(n)=\bigWedge_{i<n}a(i)$ for $n\in\Ni$.  These will all be zero
 iff TC holds.  In Lemma~\ref{lem-aa} we will check that $\ep(n)$ is
 idempotent, which allows us to define
 $\bCL=\colim_{n<\infty}\CL/\ep(n)$.  This will be our main object of
 study.  
\end{remark}

\begin{definition}\lbl{defn-ideal}
 Let $\CR$ be an ordered semiring.  An \emph{ideal} in $\CR$ is a
 subset $\CI\sse\CR$ such that 
 \begin{itemize}
  \item $0\in\CI$
  \item For all $x,y\in\CI$ we have $x\Wedge y\in\CI$
  \item For all $x\in\CR$ and $y\in\CI$ we have $x\Smash y\in\CI$. 
 \end{itemize}
\end{definition}

\begin{remark}\lbl{rem-ideal}
 Let $S$ be any subset of $\CR$, and let $\CI$ be the set of elements
 $x\in\CR$ that can be expressed in the form 
 $x = \bigWedge_{i=1}^n y_i\Smash z_i$ for some $n\in\N$ and
 $y\in\CR^n$ and $z\in S^n$.  (This should be interpreted as $x=0$ in
 the case $n=0$.)  Just as in the case of ordinary rings, this is the
 smallest ideal containing $S$, or in other words, the ideal generated
 by $S$.
\end{remark}

\begin{lemma}\lbl{lem-ideal}
 Let $\CR$ be an ordered semiring.
 \begin{itemize}
  \item[(a)] Suppose that every ideal in $\CR$ has a least upper
   bound; then $\CR$ is complete.
  \item[(b)] Suppose that $\CR$ is complete, and that for $x\in\CR$
   and every ideal $\CI\sse\CR$ we have
   $x\Smash\bigWedge\CI=\bigWedge(x\Smash\CI)$; then $\CR$ is
   completely distributive.
 \end{itemize}
\end{lemma}
\begin{proof}\leavevmode
 \begin{itemize}
  \item[(a)] Let $S$ be a subset of $\CR$, and let $\CI$ be the ideal
   that it generates.  It is then easy to see that the upper bounds
   for $\CI$ are the same as the upper bounds for $S$, and $\CI$ has a
   least upper bound by assumption, so this is also a least upper
   bound for $S$.
  \item[(b)] Now suppose we also have an element $x\in\CR$, and that
   $x\Smash\bigWedge\CI=\bigWedge(x\Smash\CI)$.  We find that
   $x\Smash\CI$ is the same as the ideal generated by $x\Smash S$, so 
   \[ \bigWedge(x\Smash S) = \bigWedge(x\Smash\CI) = 
       x\Smash\bigWedge\CI = x\Smash\bigWedge S,
   \]
   as required.
 \end{itemize}
\end{proof}

We next define two canonical subsemirings for any ordered semiring
$\CR$.  This is an obvious axiomatic generalisation of work that
Bousfield did for $\CL$ in~\cite{bo:bas}.

\begin{definition}\lbl{defn-complement}
 Let $\CR$ be an ordered semiring, and let $x$ and $y$ be elemnts of
 $\CR$.  We say that $y$ is a \emph{complement} for $x$ (and
 \emph{vice versa}) if $x\Wedge y=1$ and $x\Smash y=0$.  If such a $y$
 exists, we say that $x$ is \emph{complemented}.
\end{definition}

\begin{lemma}\lbl{lem-complement}
 If $x$ has a complement then it is unique, and we have $x\Smash x=x$.
\end{lemma}
\begin{proof}
 Let $y$ be a complement for $x$.  Multiplying the equation
 $x\Wedge y=1$ by $x$ and using $x\Smash y=0$ gives $x\Smash x=x$.

 Now let $z$ be another complement for $x$.  Multiplying the relation
 $x\Wedge y=1$ by $z$ gives $y\Smash z=z$.  Multiplying the relation
 $x\Wedge z=1$ by $y$ gives $y\Smash z=y$.  Comparing these gives
 $y=z$. 
\end{proof}

This validates the following:
\begin{definition}\lbl{defn-neg}
 For any complemented element $x$, we write $\neg x$ for the complement.
\end{definition}

\begin{definition}\lbl{defn-dist}
 For any ordered semiring $\CR$, we put 
 \begin{align*}
  \CRl &= \{x\in\CR\st x\Smash x = x\} \\
  \CRb &= \{x\in\CR\st x \text{ is complemented }\}.                      
 \end{align*}
\end{definition}

\begin{remark}\lbl{rem-dist-functor}
 Let $\phi\:\CR\to\CS$ be a homomorphism of ordered semirings.  Then
 it is clear that $\phi(\CRl)\sse\CS_{\text{latt}}$.  Moreover, if $x$
 and $y$ are complements of each other in $\CR$, we find that
 $\phi(x)$ and $\phi(y)$ are complements of each other in $\CS$.  It
 follows that $\phi(\CRb)\sse\CS_{\text{bool}}$.  In other words, both
 of the above constructions are functorial.
\end{remark}

\begin{proposition}\lbl{prop-dist}
 The set $\CRl$ is a subsemiring of $\CR$.  Moreover, for
 $x,y,z\in\CRl$ we have $x\leq y\Smash z$ iff $x\leq y$ and
 $x\leq z$, so the $\Smash$ product is just the meet operation for the
 natural ordering, and this makes $\CRl$ into a
 distributive lattice.
\end{proposition}
\begin{proof}
 It is clear that $\CR_{\text{latt}}$ contains $0$ and $1$ and is
 closed under $\Smash$.  Now suppose that $x,y\in\CRl$, and put
 $z=x\Wedge y$.  Using the commutativity and distributivity of
 $\Smash$, and the idempotence of $x$ and $y$, we obtain 
 \[ z\Smash z = x\Wedge y \Wedge (x\Smash y). \]
 We can rewrite $y\Wedge(x\Smash y)$ as
 $(1\Wedge x)\Smash y=1\Smash y=y$, so $z\Smash z=x\Wedge y=z$ as
 required.  This proves that $\CRl$ is a subsemiring.

 Now suppose that $x,y,z\in\CRl$ with $x\leq y$ and $x\leq z$.  We
 then get $x=x\Smash x\leq y\Smash z$ as required.  The converse holds
 in any ordered semiring, so we see that $\Smash$ is just the meet
 operation, as claimed.
\end{proof}

\begin{proposition}\lbl{prop-bool}
 The set $\CRb$ is a subsemiring of $\CRl$ and is a boolean algebra.
\end{proposition}
\begin{proof}
 Lemma~\ref{lem-complement} shows that $\CRb\sse\CRl$.  Note also that
 if $x\in\CRb$ then $x$ is a complement for $\neg x$, so $\neg x$ lies
 in $\CRb$ as well.

 Now suppose that $x_0,x_1\in\CRb$, with complements $y_0$ and $y_1$.
 It is then easy to check that $y_0\Smash y_1$ and $y_0\Wedge y_1$ are
 complements for $x_0\Wedge x_1$ and $x_0\Smash x_1$, showing that
 $\CRb$ is closed under $\Wedge$ and $\Smash$.  It also contains $0$
 and $1$, so $\CRb$ is a subsemiring of $\CRl$.  By one of the
 standard definitions, a boolean algebra is just a distributive
 lattice in which every element has a complement, so $\CRb$ has this
 structure. 
\end{proof}

We can generalise the definition of $\neg x$ as follows.  Put
\[ A(x) = \{y \st x\Smash y = 0\}. \]
If $y$ is a complement for $x$, then it is easy to check that it is
the largest element in the set $A(x)$.  More generally, if $x$ does
not have a complement, but $A(x)$ still has a largest element, then we
can define $\neg x$ to be that largest element.  If $\CR$ is
completely distributive then we see that $\bigWedge A(x)$ is always an
element of $A(x)$ and so qualifies as $\neg x$.  In particular, this
operation is defined for all elements of the Bousfield semiring, as
was already discussed in~\cite{bo:bas}.  However, a homomorphism
$\phi\:\CR\to\CS$ need not satisfy $\phi(\neg x)=\neg\phi(x)$ in this
more general context, even if $\phi$ preserves infinite joins.  In
particular, we do not know whether the homomorphisms $\CA\to\bCL$ and
$\CL\to\bCL$ preserve negation.  Thus, although we can compute the
negation operation in $\CA$, this does not provide much information
about $\CL$, unless we restrict attention to $\CA_{\text{bool}}$.

We can generalise still further as follows:
\begin{definition}\lbl{defn-heyting}
 Let $\CR$ be an ordered semiring, and let $x$ and $z$ be elements of
 $\CR$.  Put 
 \[ A(x,z) = \{y\in\CR\st x\Smash y\leq z\}. \]
 \begin{itemize}
  \item If $A(x,z)$ has a largest element then we denote it by
   $(x\to z)$, and call it a \emph{Heyting element} for the pair
   $(x,z)$.  
  \item A \emph{strong Heyting element} for $(x,z)$ is an element
   $y\in\CR$ such that $x\Smash y\leq z\leq y$ and $x\Wedge y=1$.
 \end{itemize}
\end{definition}

A complement for $x$ is the same as a strong Heyting element for
$(x,0)$, and our more general definition of $\neg x$ is just the same
as $(x\to 0)$.

\begin{proposition}\lbl{prop-heyting}\leavevmode
 \begin{itemize}
  \item[(a)] Any strong Heyting element is a Heyting element.
  \item[(b)] If $\CR$ is completely distributive, then every pair has
   a Heyting element.
  \item[(c)] If $x$ is complemented, then $z\Wedge\neg x$ is a Heyting
   element for $(x,z)$.
  \item[(d)] Any homomorphism of ordered semirings preserves strong
   Heyting elements.
 \end{itemize}
\end{proposition}
\begin{proof}\leavevmode
 \begin{itemize}
  \item[(a)] Let $y$ be a strong Heyting element for $(x,z)$.  Then
   $x\Smash y\leq z$, so $y\in A(x,z)$.  Let $u$ be any other element
   of $A(x,z)$, so $x\Smash u\leq z$.  Multiplying the relation
   $x\Wedge y=1$ by $u$ gives 
   \[ u=(u\Smash x)\Wedge(u\Smash y) \leq z\Wedge y, \]
   but we also have $z\leq y$ (as part of the definition of a strong
   Heyting element) so $u\leq y$ as required.
  \item[(b)] Suppose that $\CR$ is completely distributive, and put
   $y=\bigWedge A(x,z)$.  Complete distributivity implies that 
   \[ x\Smash y = \bigWedge\{x\Smash u\st u\in A(x,z)\} \leq z, \]
   so $y\in A(x,z)$, and clearly $y$ is the largest element of
   $A(x,z)$. 
  \item[(c)] Let $w$ be a complement for $x$, so $w\Smash x=0$ and
   $w\Wedge x=1$.  Put $y=z\Wedge w\geq z$.  Then 
   \begin{align*}
    x\Smash y &= (x\Smash z) \Wedge (x\Smash w) = x\Smash z \\
    x\Wedge y &= z\Wedge w\Wedge x = z\Wedge 1 = 1.
   \end{align*}
   It follows that $y$ is a strong Heyting element, as claimed.
  \item[(d)] This is clear from the definitions.
 \end{itemize}
\end{proof}
If we assume that $x\Smash x=x$ for all $x$ (so that $\CR=\CRl$) then
the Heyting elements satisfy a number of additional properties, such
as $x\Smash(x\to z)=x\Smash z$.  These properties are encapsulated by
the definition of a Heyting algebra (see~\cite[Section 1.1]{jo:ss},
for example).  They do not hold automatically in our more general
context, and we have not investigated exactly how much can be rescued.

\section{The combinatorial model}

\begin{definition}\lbl{defn-big}
 We put $\Ni=\N\cup\{\infty\}$, and give this the obvious order with
 $\infty$ as the largest element.  We will say that a subset
 $S\sse\Ni$ is \emph{small} if $S\sse [0,n)$ for some $n\in\N$;
 otherwise, we will say that $S$ is \emph{big}.  We will also say that
 $S$ is \emph{cosmall} if $\Ni\sm S$ is small, or equivalently $S$
 contains $[n,\infty]$ for some $n$.
\end{definition}
\begin{definition}\lbl{defn-Nw}
 We put $\Nw=\N\cup\{\om,\infty\}$, with the ordering 
 \[ 0 < 1 < 2 < 3 < \dotsb < \om < \infty. \]
\end{definition}

\begin{definition}\lbl{defn-A}
 The set $\CA$ has elements as follows:
 \begin{itemize}
  \item For each cosmall subset $T\sse\Ni$ and each $q\in\Ni$ we have
   an element $t(q,T)\in\CA$. 
  \item For each small subset $S\subset\Ni$ and each
   $m\in\Nw$ we have an element $j(m,S)\in\CA$.
  \item For each subset $U\sse\Ni$ we have an element $k(U)\in\CA$.
 \end{itemize}
 (For the corresponding elements of the Bousfield lattice, see
 Definition~\ref{defn-basic}.) 

 We write $\CA_t$ for the subset of elements of the form $t(q,T)$, and
 similarly for $\CA_j$ and $\CA_k$, so that
 $\CA=\CA_t\amalg\CA_j\amalg\CA_k$. 

 We define commutative binary operations $\Wedge$ and $\Smash$ on $\CA$
 as follows: 
 \begin{align*}
  t(q,T)\Wedge t(q',T') &= t(\min(q,q'),T\cup T') \\
  t(q,T)\Wedge j(m',S') &= t(q,T\cup S') \\
  t(q,T)\Wedge k(U')    &= t(q,T\cup U') \\
  j(m,S)\Wedge j(m',S') &= j(\max(m,m'),S\cup S') \\
  j(m,S)\Wedge k(U')    &= \begin{cases}
                            j(m,S\cup U') & \text{ if $U'$ is small } \\
                            k(S\cup U')   & \text{ if $U'$ is big } 
                           \end{cases} \\
  k(U)  \Wedge k(U')    &= k(U\cup U') &
 \end{align*}
 \begin{align*}
  t(q,T)\Smash t(q',T') &= t(\max(q,q'),T\cap T') \\
  t(q,T)\Smash j(m',S') &= \begin{cases}
                            j(m',T\cap S') & \text{ if } q\leq m' \\
                            k(T\cap S')    & \text{ if } q>m'
                           \end{cases}  \\
  t(q,T)\Smash k(U')    &= k(T\cap U') \\
  j(m,S)\Smash j(m',S') &= k(S\cap S') \\
  j(m,S)\Smash k(U')    &= k(S\cap U') \\
  k(U)  \Smash k(U')    &= k(U\cap U').
 \end{align*}
 We also put $0=k(\emptyset)$ and $1=t(0,\Ni)$.
\end{definition}

We next give some auxiliary definitions that will help us analyse the
structure of $\CA$. 
\begin{definition}\lbl{defn-tail}
 We write $\CP$ for the ordered semiring of subsets of $\Ni$, with the
 operations $\cup$ and $\cap$, and identity elements $0=\emptyset$ and
 $1=\Ni$.  We define $\tail\:\CA\to\CP$ by 
 \[ \tail(t(q,T)) = T \hspace{4em}
    \tail(j(m,S)) = S \hspace{4em}
    \tail(k(U)) = U.
 \]
\end{definition}
\begin{remark}\lbl{rem-tail}
 Inspection of the definitions shows that
 $\tail(x\Wedge y)=\tail(x)\cup\tail(y)$ and
 $\tail(x\Smash y)=\tail(x)\cap\tail(y)$ for all $x$ and $y$.  Once we
 have checked that $\CA$ is an ordered semiring, this will mean that
 $\tail\:\CA\to\CP$ is a homomorphism of ordered semirings.
 Inspection of the definitions also shows that
 \[ 
  \tail(x) = \{i\st k(i)\Smash x\neq 0\}
           = \{i\st k(i)\Smash x=k(i)\}
           = \{i\st k(i)\leq x\}.
 \]
\end{remark}

\begin{definition}\lbl{defn-head}
 We define 
 \[ \CH = \{t(q)\st q\in\Ni\} \amalg 
          \{j(m)\st m\in\Nw\} \amalg \{k\},
 \]
 and we define $\head\:\CA\to\CH$ in the obvious way.  
\end{definition}

\begin{remark}\lbl{rem-head-ops}
 The interaction of the head map with the operations can be summarised
 as follows:
 \begin{align*}
   t(q)\Wedge t(q') &= t(\min(q,q')) &
   t(q)\Smash t(q') &= t(\max(q,q')) \\
   t(q)\Wedge j(m') &= t(q) &
   t(q)\Smash j(m') &= j(m') \text{ \textbf{or} } k \\
   t(q)\Wedge k     &= t(q) &
   t(q)\Smash k     &= k \\
   j(m)\Wedge j(m') &= j(\max(m,m')) &
   j(m)\Smash j(m') &= k \\
   j(m)\Wedge k     &= j(m) \text{ \textbf{or} } k &
   j(m)\Smash k     &= k \\
   k   \Wedge k     &= k &
   k   \Smash k     &= k.
 \end{align*}
 Because of the indeterminate rules for $t(q)\Smash j(m')$ and
 $j(m)\Wedge k$, we cannot say that $\head\:\CA\to\CH$ is a
 homomorphism of ordered semirings.
\end{remark}

\begin{definition}\lbl{defn-j-tilde}
 We put $\N_*=\{\bot\}\amalg\Nw$, and give this the obvious order with
 $\bot$ as the smallest element.  For $m\in\N_*$ and $S\sse\Ni$ we put 
 \[ \tj(m,S) = 
     \begin{cases}
      j(m,S) & \text{ if } m > \bot \text{ and } S \text{ is small } \\
      k(S)   & \text{ if } m = \bot \text{ or } S \text{ is big }.
     \end{cases}
 \]
\end{definition}

\begin{remark}\lbl{rem-j-tilde}
 The elements $\tj(m,S)$ are distinct, except that $\tj(m,S)$ is
 independent of $m$ when $S$ is big.  The operations can be rewritten
 as follows: 
 \begin{align*}
  t(q,T) \Wedge  t(q',T')   &= t(\min(q,q'),T\cup T') \\
  t(q,T) \Wedge  \tj(m',S') &= t(q,T\cup S') \\
  \tj(m,S)\Wedge \tj(m',S') &= \tj(\max(m,m'),S\cup S') \\
  t(q,T) \Smash  t(q',T')   &= t(\max(q,q'),T\cap T') \\
  t(q,T) \Smash  \tj(m',S') &= \begin{cases}
                                \tj(m',T\cap S') & \text{ if } q\leq m' \\
                                \tj(\bot,T\cap S') & \text{ if } q> m'
                               \end{cases} \\
  \tj(m,S)\Smash \tj(m',S') &= \tj(\bot,S\cap S').
 \end{align*}
\end{remark}

\begin{proposition}\lbl{prop-osr}
 $\CA$ is an ordered semiring.
\end{proposition}
\begin{proof}
 The operations are commutative by construction, and it is immediate
 from the definitions that $0\Wedge x=1\Smash x=x\Wedge x=x$ and
 $0\Smash x=0$ and $1\Wedge x=1$.  This leaves the associativity and
 distributivity axioms.  Remark~\ref{rem-tail} takes care of the
 tails, so we just need to worry about the heads.  This is just a
 lengthy but straightforward check of cases, which is most efficiently
 done using Remark~\ref{rem-j-tilde}.  (We have also coded a partial
 formalisation using Maple.)
\end{proof}

The order on $\CA$ can be made more explicit as follows:
\begin{itemize}
 \item We have $t(q,T)\leq t(q',T')$ iff
  $T\sse T'$ and $q\geq q'$.
 \item We never have $t(q,T)\leq j(m,S)$ or $t(q,T)\leq k(U)$.
 \item We have $j(m,S)\leq t(q,T)$ iff $S\sse T$.
 \item We have $j(m,S)\leq j(m',S')$ iff 
  $S\sse S'$ and $m\leq m'$.
 \item We have $j(m,S)\leq k(U)$ iff $S\sse U$ and $U$ is big.
 \item We have $k(U)\leq t(q,T)$ iff $U\sse T$.
 \item We have $k(U)\leq j(m,S)$ iff $U\sse S$.
 \item We have $k(U)\leq k(U')$ iff $U\sse U'$.
\end{itemize}

We next want to show that $\CA$ is completely distributive.  Because
of Lemma~\ref{lem-ideal}, we can concentrate on ideals in $\CA$.

\begin{definition}\lbl{defn-tht}
 Let $\CI\sse\CA$ be an ideal.  We put 
 \begin{align*}
  A &= \bigcup \{\tail(u)\st u\in\CI\} \sse\Ni \\
  Q &= \{q\in\Ni\st t(q)\in\head(\CI)\} \sse\Ni \\
  M &= \{m\in\Nw\st j(m)\in\head(\CI)\} \sse\Nw.
 \end{align*}
 We define $\tht(\CI)\in\CA$ as follows:
 \begin{itemize}
  \item[(a)] If $Q=\emptyset$ and ($A$ is big or $M$ is empty), then
   $\tht(\CI)=k(A)$.
  \item[(b)] If $Q=\emptyset$, and $A$ is small, and $M$ is nonempty
   but has no largest element, then $\tht(\CI)=j(\om,A)$.
  \item[(c)] If $Q=\emptyset$, and $A$ is small, and $M$ has a largest
   element, then $\tht(\CI)=j(\max(M),A)$.
  \item[(d)] If $Q\neq\emptyset$, then $\tht(\CI)=t(\min(Q),A)$
 \end{itemize}
\end{definition}
(It would be possible to combine cases~(b) and~(c) in the above
definition, but it is more convenient to keep them separate, because
they behave differently in various arguments that will be given
later.)

\begin{lemma}\lbl{lem-tht-k}
 We also have $A=\{i\in\Ni\st k(i)\in\CI\}$, and $k(A)$ is the least
 upper bound for $\CI\cap\CA_k$.
\end{lemma}
\begin{proof}
 If $i\in A$ then there exists $u\in\CI$ with $i\in\tail(u)$, which
 means that $k(i)\Smash u=k(i)$.  As $\CI$ is an ideal, this means
 that $k(i)\in\CI$.  Conversely, if $k(i)\in\CI$ then
 $\{i\}=\tail(k(i))\sse A$.  This proves the alternative description
 of $A$, and the second claim follows easily from that.
\end{proof}

\begin{lemma}\lbl{lem-tht-max}
 In cases~(c) and~(d) of Definition~\ref{defn-tht} we have
 $\tht(\CI)\in\CI$, and $\tht(\CI)\geq u$ for all $u\in\CI$, so
 $\tht(\CI)$ is the largest element of $\CI$.
\end{lemma}
\begin{proof}
 We first consider case~(c), and put $m_0=\max(M)$.  By the definition
 of $M$, there is a small set $S_0$ such that $j(m_0,S_0)\in\CI$.  By
 the definition of $A$ we have $S_0\sse A$.  By assumption, the set
 $A$ is small, and therefore finite.  For each $i\in A$ we have
 $k(i)\in\CI$ by Lemma~\ref{lem-tht-k}, and so the element 
 \[ \tht(\CI) = j(m_0,A) = j(m_0,S_0) \Wedge \bigWedge_{i\in A}k(i) \]
 also lies in $\CI$.  Now consider an arbitrary element $u\in\CI$.  By
 assumption we have $Q=\emptyset$, so $u$ is either $j(m,S)$ (for some
 $m\in M$ and $S\sse A$) or $k(S)$ (for some $S\sse A$).  In all cases
 it is clear that $u\leq\tht(\CI)$, as required.

 Now consider case~(d), and put $q_0=\min(Q)$.  By the definition of
 $Q$, there is a cosmall set $T_0$ such that $t(q_0,T_0)\in\CI$.  By
 the definition of $A$ we have $T_0\sse A$, and $T_0$ is cosmall, so
 $A=T_0\amalg A_0$ for some finite set $A_0\subset\N$.  For $i\in A_0$
 we have $k(i)\in\CI$ by Lemma~\ref{lem-tht-k}, so the element
 \[ \tht(\CI) = t(q_0,A) = t(q_0,T_0)\Wedge\bigWedge_{i\in A_0}k(i) \]
 also lies in $\CI$.  Now consider an arbitrary element $u\in\CI$.  If
 $u\in\CA_j\amalg\CA_k$ then $u=j(m,S)$ or $u=k(S)$ for some
 $S\sse A$, and this gives $u\leq\tht(\CI)$ (independent of the value
 of $m$).  If $u\in\CA_t$ then $u=t(q,T)$ for some $q\in Q$ and
 $T\sse A$, and we must have $q\geq\min(Q)=q_0$, which again gives
 $u\leq\tht(\CI)$ as required.
\end{proof}

\begin{lemma}\lbl{lem-tht-a}
 In case~(a) of Definition~\ref{defn-tht}, the element
 $\tht(\CI)=k(A)$ is the least upper bound for $\CI$.
\end{lemma}
\begin{proof}
 We see from Lemma~\ref{lem-tht-k} that the element $\tht(\CI)=k(A)$
 is the least upper bound for $\CI\cap\CA_k$, so we just need to check
 that it is an upper bound for all of $\CI$.  Consider an arbitrary
 element $u\in\CI$.  As $Q=\emptyset$ we must have $u=j(m,S)$ or
 $u=k(S)$ for some $S\sse A$.  As $A$ is big, it follows that
 $u\leq k(A)$ as required.
\end{proof}

\begin{lemma}\lbl{lem-tht-b}
 In case~(b) of Definition~\ref{defn-tht}, the set $M$ is infinite and
 contained in $\N$.  Moreover, we have $j(m,A)\in\CI$ for all $m\in M$, 
 and the element $\tht(\CI)=j(\om,A)$ is the least upper bound for
 $\CI$. 
\end{lemma}
\begin{proof}
 By assumption, $M$ is a nonempty subset of $\Nw$ with no largest
 element.  By inspection, this is only possible if $M$ is an infinite
 subset of $\N$.  Moreover, the set $A$ is small and therefore
 finite.  It follows using Lemma~\ref{lem-tht-k} that $k(A)\in\CI$.  
 If $m\in M$ then $j(m,S_m)\in\CI$ for some $S_m$, which must be a
 subset of $A$.  It follows that the element
 $j(m,A)=j(m,S_m)\Wedge k(A)$ also lies in $\CI$.  

 Now let $u$ be an arbitrary element of $\CI$.  As $Q=\emptyset$, we
 must have $u=j(m,S)$ for some $m\in M$ and $S\sse A$, or $u=k(S)$ for
 some $S\sse A$.  From this it is easy to check that $j(\om,A)$ is the
 least upper bound.
\end{proof}

\begin{proposition}\lbl{prop-cd}
 $\CA$ is completely distributive.
\end{proposition}
\begin{proof}
 We will use the criteria in Lemma~\ref{lem-ideal}.  Let $\CI\sse\CA$
 be an ideal.  Lemmas~\ref{lem-tht-max}, \ref{lem-tht-a}
 and~\ref{lem-tht-b} show that the element $a=\tht(\CI)$ is always a
 least upper bound for $\CI$.  It follows that $\CA$ is complete.

 Now consider an element $x\in\CA$, and put
 $\CI'=x\Smash\CI$ and $a'=\bigWedge\CI'$.
 It is clear that $x\Smash a$ is an upper bound for $\CI'$, so
 $a'\leq x\Smash a$, and we must show that this is an
 equality.  This is clear from Lemma~\ref{lem-tht-max} in cases~(c)
 and~(d) of Definition~\ref{defn-tht}, so we need only consider
 cases~(a) and~(b).

 In these cases we have $\CI'\sse\CI\sse\CA_j\cup\CA_k$, and also
 $a\in\CA_j\cup\CA_k$.  Note that  
 \[ \tail(a') = \bigcup\{\tail(u')\st u'\in\CI'\} 
     = \tail(x)\cap\tail(a) = \tail(x\Smash a),
 \]
 so we just need to worry about the head.  

 Now suppose that $x$ also lies in $\CA_j\cup\CA_k$.  From the
 definitions we have
 \[ (\CA_j\cup\CA_k)\Smash(\CA_j\cup\CA_k) = \CA_k, \]
 and it follows that $\head(a')=k=\head(x\Smash a)$ as required.

 Now suppose instead that $x=t(q,T)$. 

 In case~(a) we then have $x\Smash a=k(T\cap A)$, and $T\cap A$ is big
 (because $A$ is big and $T$ is cosmall).  Using Lemma~\ref{lem-tht-k}
 we see that $k(i)\in x\Smash\CI$ for all $i\in T\cap A$, and it
 follows that $a'\geq k(T\cap A)=x\Smash a$ as required. 

 Finally, consider case~(b) (still with $x=t(q,T)$).  Put
 $M'=\{m'\in M\st m'\geq q\}$.  Using Lemma~\ref{lem-tht-b} we see
 that $M'$ is an infinite subset of $\N$, and that $j(m',A)\in\CI$ for
 all $m'\in M'$.  In this context we have
 $x\Smash j(m',A)=j(m',A\cap T)$.  It follows that
 $a'\geq j(m',A\cap T)$ for all $m'\in M'$, and thus that
 $a'\geq j(\om,A\cap T)=x\Smash a$ as required.
\end{proof}

\begin{proposition}\lbl{prop-CAd}
 $\CA_{\text{latt}} = \CA_t\amalg\CA_k$.
\end{proposition}
\begin{proof}
 Just inspect the definitions to see which elements satisfy
 $x\Smash x=x$.
\end{proof}

\begin{proposition}\lbl{prop-CAb}
 We have
 \[ \CA_{\text{bool}} = 
     \{t(0,T)\st T \text{ is cosmall }\} \amalg 
     \{k(U) \st U\text{ is small }\},
 \]
 with $\neg t(0,T)=k(\Ni\sm T)$ and $\neg k(U)=t(0,\Ni\sm U)$. 
\end{proposition}
\begin{proof}
 Inspection of the definitions shows that when $U\sse\Ni$ is small, we
 have $t(0,\Ni\sm U)\Wedge k(U)=t(0,\Ni)=1$ and
 $t(0,\Ni\sm U)\Smash k(U)=k(\emptyset)=0$.  Thus, the claimed
 elements all lie in $\CA_{\text{bool}}$.  Conversely, suppose that
 $x$ and $y$ are complementary elements of $\CA_{\text{bool}}$.  We
 must then have $x\Wedge y=1=t(0,\Ni)$.  Inspection of the definitions
 shows that this is only possible if one of $x$ and $y$ has the form
 $t(0,T)$ for some cosmall $T$; we may assume without loss that
 $x=t(0,T)$.  We must also have $t(0,T)\Smash y=0$, and this is only
 possible if $y=k(U)$ with $U\cap T=\emptyset$.  The condition
 $x\Wedge y=1$ now reduces to $T\cup U=\Ni$, so we must have
 $U=\Ni\sm U$.
\end{proof}

\begin{remark}\lbl{rem-heyting-A}
 It is also possible to tabulate all the Heyting elements $(x\to y)$
 for $x,y\in\CA$, and to determine which of them are strong.  Strong
 Heyting elements in $\CA$ will give strong Heyting elements in
 $\bCL$, but the same cannot be guaranteed for weak Heyting elements.
 The complete tabulation involves a rather long list of cases, so we
 will not give all the details here.
\end{remark}

\section{Basic Bousfield classes}
 
We now introduce notation for various spectra, and the corresponding
Bousfield classes.  The names that we will use for some of these
classes are the same as the names of elements of $\CA$.  Later we will
consider the map $\phi\:\CA\to\CL$ that sends each element of $\CA$ to
the element of $\CL$ with the same name.

\begin{definition}\lbl{defn-basic}\leavevmode
 \begin{itemize}
  \item For $n\in\N$ we let $K(n)$ denote the $n$'th Morava
   $K$-theory~\cite{jowi:bpo}.  In particular, $K(0)$ is the rational
   Eilenberg-MacLane spectrum $H\Q$.  We also write $K(\infty)$ for
   the mod $p$ Eilenberg-MacLane spectrum, and $k(n)=\bc{K(n)}$.
  \item For any subset $U\sse\Ni$ we put $K(U)=\bigWedge_{i\in U}K(i)$
   and $k(U)=\bc{K(U)}$.
  \item It is a theorem of Mitchell~\cite{mi:fca} that for each
   $n\in\N$ we can choose a ($p$-locally) finite spectrum $U(n)$ of type
   $n$, meaning that $K(i)_*U(n)=0$ iff $i<n$.  We choose $U(0)$ to be
   $S$ and $U(1)$ to be the Moore spectrum $S/p$.  We put
   $F(n)=F(U(n),U(n))$, which is a self-dual finite ring spectrum of
   type $n$.  Note that $F(0)=S^0$.  In all cases we put
   $f(n)=\bc{F(n)}$.  As a well-known consequence of the Thick
   Subcategory Theorem~\cite[Theorem 7]{hosm:nshii}, these 
   Bousfield classes do not depend on the choice of $U(n)$.
  \item For $q\in\N$ we recall that the Bott periodicity isomorphism
   $\Om SU=BU$ gives a natural virtual vector bundle over $\Om SU(p^q)$,
   and the associated Thom spectrum $X(p^q)$ has a natural ring
   structure.  The $p$-localisation of this has a $p$-typical summand
   called $T(q)$ (see~\cite[Section 6.5]{ra:ccs}).  We will also take
   $T(\infty)=BP$.  Note that $T(0)$ is just the ($p$-local) sphere
   spectrum $S$.  In all cases we put $t(q)=\bc{T(q)}$ and
   $t(q;n)=t(q)\Smash f(n)$.  
  \item Now suppose we have $q\in\N$ and a cosmall set $T\sse\Ni$.
   For any $n$ such that $[n,\infty]\sse T$, we define
   $t(q,T;n)=t(q;n)\Wedge k(T)$.  We also define $t(q,T)=t(q,T;n_0)$,
   where $n_0$ is the smallest integer such that $[n_0,\infty]\sse T$.
  \item For $m\in\Ni$ we let $J(m)$ denote the Brown-Comenetz
   dual of $T(m)$, so there is a natural isomorphism
   \[ [X,J(m)] \simeq \Hom(\pi_0(T(m)\Smash X),\QZpl) \]
   for all spectra $X$.  We also put $J(\om)=\bigWedge_{m\in\N}J(m)$,
   and $j(m)=\bc{J(m)}$ for all $m\in\Nw$.  Given a small
   set $S$, we put $j(m,S)=j(m)\Wedge k(S)$.
  \item For $n\in\N$ we choose a good $v_n$ self-map
   $w_n$ of $U(n)$.  (Here we use Definition~4.5 from~\cite{host:mkl},
   which is a slight modification of definitions used in
   \cite{hosm:nshii,de:srs}.  This means that
   $w_n\Smash 1=1\Smash w_n$ as endomorphisms of $U(n)\Smash U(n)$,
   and that $1_{BP}\Smash w_n=v_n^{p^{d_n}}\Smash 1_{U(n)}$ as
   endomorphisms of $BP\Smash U(n)$ for some $d_n\geq 0$.)  We also
   write $w_n$ for the corresponding element of $\pi_*(F(n))$, and we
   put $K'(n)=F(n)[w_n^{-1}]$ and $k'(n)=\bc{K'(n)}$.
  \item Now fix $n\in\N$.  Let $L_n$ denote the Bousfield localisation
   functor with respect to the Johnson-Wilson spectrum $E(n)$, and let
   $C_nX$ denote the fibre of the natural map $X\to L_nX$.  We also put
   $A(n)=C_nK'(n)$ and $a(n)=\bc{A(n)}$.  Note here that the Smash Product
   Theorem~\cite[Theorem 7.5.6]{ra:nps} gives
   $A(n)=K'(n)\Smash C_nS$.  We also put $\ep(n)=\bigWedge_{i<n}a(i)$
   for all $n\in\Ni$.
 \end{itemize}
\end{definition}

\begin{remark}\lbl{rem-tc}
 The original formulation of Ravenel's Telescope
 Conjecture~\cite[Conjecture 10.5]{ra:lrc} says that $k'(n)=k(n)$ for
 all $n\in\N$.  It is shown in~\cite[Section 1.3]{marash:tls} that
 this is equivalent to the claim that $K'(n)=L_nK'(n)$, which is in
 turn equivalent to $a(n)=0$.  These equivalences can also be obtained
 from Lemma~\ref{lem-at} below.  The formulation $a(n)=0$ is also used
 in~\cite{ho:blf,hopa:sbl}.  We can reformulate it again as
 $\ep(n)=0$ for all $n\in\Ni$, or as $\ep(\infty)=0$.
\end{remark}

\begin{remark}\lbl{rem-notation}
 We offer some translations between our notation and that used by some
 other authors.
 \begin{itemize}
  \item[(a)] In~\cite[Section 3]{ra:lrc}, Ravenel uses the notation
   $X_n$ for what we have called $X(p^n)$.  He only mentions $T(n)$ in
   passing, but he calls it $T_n$.  In~\cite{ra:ccs}
   and~\cite{ra:nps}, however, Ravenel uses the same notation as we do
   here.
  \item[(b)] We have used the symbol $k(n)$ for the Bousfield class of
   the spectrum $K(n)$, with homotopy ring $\Z/p[v_n,v_n^{-1}]$.
   However, many other sources use the symbol $k(n)$ for a certain
   spectrum with homotopy ring $\Z/p[v_n]$, whose Bousfield class is
   different from that of $K(n)$.  We will instead use the notation
   $BP\ip{n}/I_n$ for this spectrum.
  \item[(c)] Our finite spectra $U(n)$ and $F(n)$ have type $n$, and
   they have the same Bousfield class as any other finite spectrum of
   type $n$.  In particular, this applies to the Toda-Smith spectra
   when they exist.  The Toda-Smith spectrum of type $n$ is
   traditionally denoted $V(n-1)$, but we will call it $S/I_n$. 
  \item[(d)] Our class $k'(n)$ is often denoted $\text{Tel}(n)$ or
   $T(n)$.  Our notation is chosen to reflect the fact that
   $k'(n)=k(n)$ modulo the telescope conjecture.
 \end{itemize}
\end{remark}

\begin{remark}\lbl{rem-y}
 The paper~\cite{marash:tls} is an incomplete attempt to disprove TC.
 It involves spectra called $y(n)$ and $Y(n)$, which we will not
 define here.  In Section~3 of that paper, the authors say (in our
 notation) that $y(n)$ might be the same as $T(n)\Smash S/I_n$ in
 cases where $S/I_n$ exists, and some of their calculations provide
 evidence for that possibility.  As a closely related possibility, it
 might be that $\bc{y(n)}=t(n)\Smash f(n)$ as Bousfield classes for
 all $n$.  This would give $\bc{Y(n)}=t(n)\Smash k'(n)$.  If the
 strategy in~\cite{marash:tls} could be completed, it would show that
 $A(n)\Smash y(n)\neq 0$ for all $n>1$.  If we also knew that
 $\bc{y(n)}=t(n)\Smash f(n)$, we could conclude that
 $t(n)\Smash a(n)\neq 0$ for $n>1$.  On the other hand, it is known
 $t(i)<t(j)$ whenever $i>j$, and that $t(\infty)\Smash a(n)=0$.  One
 would thus want to ask whether $t(n+1)\Smash a(n)$ is zero or not.
\end{remark}

\begin{definition}\lbl{defn-phi}
 We define $\phi\:\CA\to\CL$ to be the map that sends each element of
 $\CA$ to the element of $\CL$ with the same name.
\end{definition}

\begin{definition}\lbl{defn-bCL}
 Later we will prove that $\ep(n)$ is idempotent for all $n$.
 Assuming this for the moment, we can define
 \[ \bCL=\colim_{n<\infty}\CL/\ep(n). \]
 We write $\pi$ for the canonical quotient map $\CL\to\bCL$, and we
 put $\bphi=\pi\phi\:\CA\to\bCL$.
\end{definition}

We will need some properties of the spectra $T(q)$.

\begin{lemma}\lbl{lem-tq-finite}
 The spectrum $T(q)$ is $(-1)$-connected, and each homotopy group is
 finitely generated over $\Zpl$.
\end{lemma}
\begin{proof}
 As $X(p^q)$ is the Thom spectrum of a virtual bundle of virtual
 dimension zero, it is certainly $(-1)$-connected.  It is a standard
 calculation that
 \[ H_*(X(p^q))=\Z[b_i\st 0<i\leq p^q], \]
 with $|b_i|=2i$.  Using this and the Atiyah-Hirzebruch spectral
 sequence 
 \[ H_i(X(p^q);\pi_j(S)) \convto \pi_{i+j}(X(p^q)), \]
 we see that the homotopy groups of $X(p^q)$ are finitely generated
 over $\Z$.  As $T(q)$ is a summand in $X(p^q)_{(p)}$, we deduce that
 it is $(-1)$-connected, with homotopy groups that are finitely
 generated over $\Zpl$. 
\end{proof}

\begin{lemma}\lbl{lem-algebroid}
 For $q\geq r$ we have
 \[ T(q)_*T(r) = T(q)_*[t_1,\dotsc,t_r] \]
 (with $|t_i|=2(p^i-1)$).
\end{lemma}
The literature contains various similar and closely related results,
but we have not been able to find this precise version.
\begin{proof}
 By construction~\cite[Section 6.5]{ra:ccs}, there is a map
 $i_q\:T(q)\to BP$ which induces an isomorphism from $BP_*T(q)$ to
 the subring $BP_*[t_1,\dotsc,t_q]$ of the ring
 $BP_*BP=BP_*[t_i\st i>0]$.  This implies that the connectivity of the
 map $i_q$ is $|t_{q+1}|-1$, which is strictly greater than $|t_r|$.
 The connectivity of the map 
 \[ i_q\Smash 1\:T(q)\Smash T(r) \to BP\Smash T(r) \]
 is at least as large as that of $i_q$, so the elements
 $t_i\in BP_*T(r)$ have unique preimages in $T(q)_*T(r)$, which
 we also denote by $t_i$.  These give us a map 
 \[ \al \: T(q)_*[t_1,\dotsc,t_r] \to T(q)_*T(r). \]
 From the description of $BP_*T(r)$ it follows easily that
 $H_*(T(r))=\Zpl[t_1,\dotsc,t_r]$, so we have an Atiyah-Hirzebruch
 spectral sequence
 \[ H_*(T(r);T(q)_*) = 
     T(q)_*[t_1,\dotsc,t_r] \convto T(q)_*T(r).
 \]
 The map $\al$ provides enough permanent cycles to show that the
 spectral sequence collapses, and it follows that $\al$ is an
 isomorphism. 
\end{proof}

\begin{lemma}\lbl{lem-tower}
 If $m\leq m'\leq\infty$, then $T(m)$ can be expressed as the homotopy
 inverse limit of a tower of spectra $Q(r)$, where the fibre of the
 map $Q(r+1)\to Q(r)$ is a product of suspended copies of $T(m')$, and
 $Q(r)=0$ for $r<0$.
\end{lemma}
\begin{proof}
 This is essentially a standard construction with generalised Adams
 resolutions.  

 Let $j\:M\to S$ denote the fibre of the unit map $S\to T(m')$ (where
 $S$ denotes the $p$-local sphere).  Recall that $H_*T(m')$ is a
 polynomial ring over $\Zpl$ with generators $t_i$ in degree
 $2(p^i-1)$ for $1\leq i\leq m'$.  It follows that $M$ is
 $(d-1)$-connected, where $d=2(p-1)>0$.  Now put
 $N(r)=M^{(r)}\Smash T(m)$.  We can use $j$ to make these into a
 tower.  We let $P(r)$ denote the cofibre of the map $N(r+1)\to N(r)$,
 which is $T(m')\Smash M^{(r)}\Smash T(m)$.  We also let $Q(r)$ denote
 the cofibre of the map $N(r)\to N(0)=T(m)$.  Connectivity arguments
 show that $T(m)$ is the homotopy inverse limit of the spectra $Q(r)$.
 We know from Lemma~\ref{lem-algebroid}  
 \begin{align*}
  T(m')_*T(m') &= T(m')_*[t_i\st 1\leq i\leq m'] \\
  T(m')_*T(m)  &= T(m')_*[t_i\st 1\leq i\leq m].
 \end{align*}
 It follows that the spectra $T(m')\Smash T(m')$ and
 $T(m')\Smash T(m)$ are free modules over $T(m')$, and thus that the
 same is true of $P(r)$.  This means that
 $P(r)=\bigWedge_i\Sg^{d_i}T(m')$ for some sequence $(d_i)$.  It is
 also easy to see that $P(r)$ is of finite type, so $d_i\to\infty$, so
 $P(r)$ can also be described as $\prod_i\Sg^{d_i}T(m')$.  Note also
 that the fibre of the map $Q(r+1)\to Q(r)$ is the same as $P(r)$, by
 the octahedral axiom.
\end{proof}

We will also need the following fact about $K'(n)$:
\begin{lemma}\lbl{lem-kk-ring}
 $K'(n)$ admits a ring structure such that the natural map
 $F(n)\to K'(n)$ is a ring map. 
\end{lemma}
The standard way to prove this is to show that $K'(n)$ is a Bousfield
localisation of $F(n)$.  We will give essentially the same argument,
formulated in a more direct way.
\begin{proof}
 If $Y$ is a finite spectrum of type $n+1$ then $1_{DY}\Smash w_n$
 induces a nilpotent endomorphism of $MU\Smash (DY\Smash F(n))$, so
 the Nilpotence Theorem tells us that $1_{DY}\Smash w_n$ is itself
 nilpotent, which implies that the spectrum
 $F(Y,K'(n))=DY\Smash F(n)[w_n^{-1}]$ is zero.  

 Now let $Q$ be the cofibre of the natural map $F(n)\to K'(n)$.  It is
 not hard to see that this is a homotopy colimit of spectra isomorphic
 to $F(n)/w_n^k$, which are finite and of type $n+1$.  Using this, we
 see that $F(Q,K'(n))=0$.  It follows inductively that the restriction
 map 
 \[ F(K'(n)^{(r)},K'(n)) \to F(F(n)^{(r)},K'(n)) \]
 are isomorphisms for all $r\geq 0$.  Using the case $r=2$, we see
 that the map 
 \[ F(n) \Smash F(n) \xra{\text{mult}} F(n) \xra{} K'(n) \]
 extends in a unique way over $K'(n)\Smash K'(n)$.  Using the cases
 $r=3$ and $r=1$, we see that this extension gives an associative and
 unital product.
\end{proof}

\section{Relations in \texorpdfstring{$\CL$}{L}}

We first recall some basic general facts about Bousfield classes:
\begin{proposition}\lbl{prop-bc-omni}\leavevmode
 \begin{itemize}
  \item[(a)] If $R$ is a ring spectrum then
   $\bc{R}\Smash\bc{R}=\bc{R}$.  Moreover, if $M$ is any $R$-module
   spectrum then $\bc{M}=\bc{R}\Smash\bc{M}\leq\bc{R}$.
  \item[(b)] Let $K$ be a ring spectrum such that all nonzero
   homogeneous elements of $K_*$ are invertible.  Then for any $X$ we
   have either $K_*X=0$ and $\bc{K}\Smash\bc{X}=0$, or $K_*X\neq 0$
   and $\bc{K}\Smash\bc{X}=\bc{K}$ and $\bc{X}\geq\bc{K}$.
  \item[(c)] Let $X$ be a spectrum, and let $v\:\Sg^dX\to X$ be a
   self-map with cofibre $X/v$ and telescope $X[v^{-1}]$.  Then
   $\bc{X}=\bc{X/v}\Wedge\bc{X[v^{-1}]}$.
  \item[(d)] Let $T$ and $X$ be spectra such that the homotopy groups
   of $X$ are finitely generated over $\Zpl$.  Then $T\Smash IX=0$ iff
   $T\Smash I(X/p)=0$ iff $F(T,X/p)=0$.
  \item[(e)] Suppose again that the homotopy groups
   of $X$ are finitely generated over $\Zpl$, and that they are not
   all torsion groups.  Then 
   $\bc{X}=\bc{X^\wedge_p}=\bc{H\Q}\Wedge\bc{X/p}$.
 \end{itemize}
\end{proposition}
\begin{proof}
 None of this is new, but we will give brief proofs for the
 convenience of the reader.
 \begin{itemize}
  \item[(a)] It is immediate from the definitions that
   $\bc{X\Smash Y}\leq\bc{X}$ and $\bc{X\Smash Y}\leq\bc{Y}$.
   Similarly, it is clear that $\bc{X}\leq\bc{Y}$ whenever
   $X$ is a retract of $Y$.  If $R$ is a ring and $M$ is an $R$-module
   then $M$ is a retract of $R\Smash M$ (via the unit map
   $\eta\Smash 1\:M\to R\Smash M$ and the multiplication
   $R\Smash M\to M$) so $\bc{M}\leq\bc{R\Smash M}$.  On the other
   hand, we have $\bc{R\Smash M}\leq\bc{R}$ and
   $\bc{R\Smash M}\leq\bc{M}$.  Putting this together gives
   $\bc{M}=\bc{R}\Smash\bc{M}\leq\bc{R}$ as claimed.  Taking $M=R$
   gives $\bc{R}\Smash\bc{R}=\bc{R}$.  This is all covered
   by~\cite[Section 2.6]{bo:bas} and~\cite[Proposition 1.24]{ra:lrc}.
  \item[(b)] A slight adaptation of standard linear algebra shows that
   all graded modules over $K_*$ are free.  If $M$ is a $K$-module
   then we can choose a basis $\{e_i\}_{i\in I}$ for $M_*$ over $K_*$,
   and this will give a map $f\:\bigWedge_{i\in I}\Sg^{|e_i|}K\to M$
   of $K$-modules such that $\pi_*(f)$ is an isomorphism, which means
   that $f$ is an equivalence.  Thus, if $M_*\neq 0$ then
   $\bc{M}=\bc{K}$.  Taking $M=K\Smash X$ gives claim~(b).  This is
   all covered in~\cite[Section 1.3]{hosm:nshii}.
  \item[(c)] First note that if $X=0$ then it is clear that $X/v=0$
   and $X[v^{-1}]=0$.  Conversely, if $X/v=0$ then $v$ is an
   equivalence, so $X[v^{-1}]=X$; so if $X[v^{-1}]$ is also $0$, then
   $X=0$.  Thus, we have $X=0$ iff $X/v=X[v^{-1}]=0$.  Now let $T$ be
   an arbitrary spectrum, and put 
   \[ w=1_T\Smash v\:T\Smash X\to T\Smash X, \]
   so $T\Smash(X/v)=(T\Smash X)/w$ and
   $T\Smash X[v^{-1}]=(T\Smash X)[w^{-1}]$.  By applying our first
   claim to $w$, we see that $T\Smash X=0$ iff
   $T\Smash(X/v)=T\Smash X[v^{-1}]=0$.  In other words, we have
   $\bc{X}=\bc{X/v}\Wedge\bc{X[v^{-1}]}$ as claimed.  This
   is~\cite[Lemma 1.34]{ra:lrc}.
  \item[(d)] First, we have $\pi_k(IX)=\Hom(\pi_{-k}(X),\QZpl)$.
   Using the fact that $\pi_{-k}(X)$ is finitely generated, we see
   that this is a torsion group.  It follows that $(IX)[p^{-1}]=0$,
   so~(c) gives $\bc{IX}=\bc{(IX)/p}$.  On the other hand, $I$
   converts cofibrations to fibrations (with arrows reversed), giving
   $(IX)/p=\Sg I(X/p)$, so $\bc{IX}=\bc{I(X/p)}$, so
   $T\Smash IX=0$ iff $T\Smash I(X/p)=0$.  Next, we note that each
   homotopy group $\pi_k(X/p)$ is finite, which implies that the
   natural map 
   \[ \pi_k(X/p) \to \Hom(\Hom(\pi_k(X/p),\QZpl),\QZpl) \]
   is an isomorphism, so the natural map $X/p\to I^2(X/p)$ is an
   equivalence.  This gives 
   \[ \pi_kF(T,X/p) = \pi_kF(T,I^2(X/p)) =
       \Hom(\pi_{-k}(T\Smash I(X/p)),\QZpl).
   \]
   It is well-known that for an abelian group $A$ we have $A=0$ iff
   $\Hom(A,\QZpl)=0$, so $F(T,X/p)=0$ iff $T\Smash I(X/p)=0$, as
   claimed.  (This is essentially covered
   by~\cite[Section 2]{ra:lrc}.) 
  \item[(e)] As a special case of~(c) we have
   $\bc{X}=\bc{X[p^{-1}]}\Wedge\bc{X/p}$.  As everything is implicitly
   $p$-local we see that $X[p^{-1}]$ is a module over $S[p^{-1}]=S\Q=H\Q$,
   with homotopy groups $\pi_*(X)\ot\Q\neq 0$, so
   $\bc{X[p^{-1}]}=\bc{H\Q}$, so $\bc{X}=\bc{H\Q}\Wedge\bc{X/p}$.  Now
   let $Y$ denote the $p$-completion of $X$, which can be constructed
   as the cofibre of the natural map $F(S\Q,X)\to X$.  As $X$ is
   assumed to have finite type, we just have
   $\pi_*(Y)=\Z_p\ot\pi_*(X)$, and this is again not a torsion group,
   so $\bc{Y}=\bc{H\Q}\Wedge\bc{Y/p}$.  Moreover, as $F(S\Q,X)$ is a
   module over $S\Q$ we see that $F(S\Q,X)/p=0$ and so $X/p=Y/p$,
   which gives $\bc{Y}=\bc{X}$.
 \end{itemize}
\end{proof}

We next recall some relations between the elements named in
Definition~\ref{defn-basic}.  Again, many of these results are in the
literature, but it seems useful to collect proofs in one place.

\begin{lemma}\lbl{lem-Kn-mod}
 For any $n\in\Ni$ and any spectrum $X$, we have either $K(n)_*X=0$
 and $k(n)\Smash\bc{X}=0$, or $K(n)_*X\neq 0$ and
 $k(n)\Smash\bc{X}=k(n)$ and $\bc{X}\geq k(n)$.
\end{lemma}
\begin{proof}
 This is a standard instance of Proposition~\ref{prop-bc-omni}(b).
\end{proof}

\begin{lemma}\lbl{lem-kk}
 For all $i$ we have $k(i)\Smash k(i)=k(i)$, and $k(i)\Smash k(i')=0$
 for $i\neq i'$.  Thus $k(U)\Smash k(U')=k(U\cap U')$ and
 $k(U)\Wedge k(U')=k(U\cup U')$ (for all $U,U'\sse\Ni$).
\end{lemma}
\begin{proof}
 The first claim holds because $K(i)$ is a ring spectrum, and the
 second can be deduced from the fact that over $K(i)_*K(i')$ we have
 two isomorphic formal group laws of different heights.  It is also
 proved as part of~\cite[Theorem 2.1]{ra:lrc}.  The remaining claims
 are clear from the first two. 
\end{proof}

\begin{lemma}\lbl{lem-rings-idempotent}
 The elements $t(q)$, $f(n)$, $t(q;n)$, $k(i)$, $k(U)$ and $k'(i)$ all
 satisfy $u\Smash u=u$.
\end{lemma}
\begin{proof}
 We have already seen the cases $k(i)$ and $k(U)$.  The spectra $T(q)$
 and $F(n)$ have ring structures by construction, and $K'(n)$ is also
 a ring by Lemma~\ref{lem-kk-ring}, so all remaining claims follow
 from Proposition~\ref{prop-bc-omni}(a).
\end{proof}

\begin{lemma}\lbl{lem-kf}
 For all $i\in\Ni$ and $n\in\N$ we have $k(i)\Smash f(n)=0$ if
 $i<n$, and $k(i)\Smash f(n)=k(i)$ if $i\geq n$.
\end{lemma}
\begin{proof}
 The spectrum $F(n)$ was defined to have type $n$, which means by
 definition that $K(i)_*F(n)=0$ iff $i<n$.  The claim follows from
 this together with Lemma~\ref{lem-Kn-mod}.
\end{proof}

\begin{lemma}\lbl{lem-tt}
 For all $q\leq q'\leq\infty$ we have $t(q)\geq t(q')$, and
 $t(q)\Smash t(q')=t(q')$.
\end{lemma}
\begin{proof}
 There is a morphism $T(q)\to T(q')$ of ring spectra, which makes
 $T(q')$ into a $T(q)$-module spectrum.
\end{proof}

\begin{lemma}\lbl{lem-tk}
 For all $q,i\in\Ni$ we have $t(q)\geq k(i)$, and
 $t(q)\Smash k(i)=k(i)$.
\end{lemma}
\begin{proof}
 There is a morphism $T(q)\to T(\infty)=BP\to K(i)$ of ring
 spectra.  
\end{proof}

\begin{corollary}\lbl{cor-tk}
 For all $q\in\Ni$ and $n\in\N$ and $U\sse\Ni$ we have
 $t(q;n)\Smash k(U)=k(U\cap[n,\infty])$. 
\end{corollary}
\begin{proof}
 This is clear from Lemmas~\ref{lem-kf} and~\ref{lem-tk}.
\end{proof}

\begin{lemma}\lbl{lem-above}
 Let $X\in\CB$ be such that $\pi_i(X)$ is torsion for all $i$, and
 $\pi_i(X)=0$ for $i>0$.  Then $\bc{X}\leq k(\infty)$.
\end{lemma}
\begin{proof}
 Put 
 \[ \CC=\{X\st \bc{X}\leq k(\infty)\} =
     \{X\st X\Smash Z=0 \text{ whenever } K(\infty)\Smash Z=0\}.
 \]
 This is closed under cofibres, coproducts and retracts, and it
 follows that it is closed under homotopy colimits of sequences.  It
 contains $K(\infty)=H\Z/p$ by definition, so it contains $HA$
 whenever $pA=0$ (by coproducts), so it contains $HA$ whenever
 $p^dA=0$ (by cofibres), so it contains $HA$ whenever $A$ is torsion
 (by sequential colimits).  Thus, if $X$ is a torsion spectrum, we see
 that all the Postnikov sections $X[-d]=\Sg^{-d}H(\pi_{-d}X)$ lie in
 $\CC$, so $X[-d,0]\in\CC$ for all $d\geq 0$ (by induction and
 cofibres), so $X=X[-\infty,d]\in\CC$ (by sequential colimits).
\end{proof}

\begin{lemma}\lbl{lem-jh}
 For all $m\in\Nw$ we have $j(m)\leq k(\infty)$.
\end{lemma}
\begin{proof}
 For $m\neq\om$ we have $J(m)=IT(m)$, and $T(m)$ is $(-1)$-connected 
 with finitely generated homotopy groups, so Lemma~\ref{lem-above}
 applies to $J(m)$.  As $J(\om)=\bigWedge_{i\in\N}J(i)$, the claim
 holds for $m=\om$ as well.
\end{proof}

\begin{lemma}\lbl{lem-jj-leq}
 If $m\leq m'\in\Nw$, then $j(m)\leq j(m')$.
\end{lemma}
\begin{proof}
 The case $m'=\om$ is immediate from the definition of $J(\om)$, and
 the case $m=\om$ will follow from the cases $m\in\N$, so we may
 assume that $\om\not\in\{m,m'\}$.

 We must show that if $X\Smash J(m')=0$ then $X\Smash J(m)=0$.  In
 view of Lemma~\ref{lem-tq-finite}, we can translate these statements
 using part~(d) of Proposition~\ref{prop-bc-omni}.  We must now show
 that if $F(X,T(m')/p)=0$ then $F(X,T(m)/p)=0$.  This translated statement
 follows easily from Lemma~\ref{lem-tower}.
\end{proof}

\begin{lemma}\lbl{lem-tj}
 For all $m\in\Nw$ and $q\in\Ni$ and $n\in\N$ we have
 $j(m)\leq k(\infty)\leq t(q;n)$.  Moreover, if $m<q$ then we
 have $t(q)\Smash j(m)=0$, but if $m\geq q$ then $t(q)\Smash j(m)=j(m)$.
\end{lemma}
Most of the statements with $m=0$ are contained
in~\cite[Lemma 7.1]{hopa:sbl}.
\begin{proof}
 We know from Lemma~\ref{lem-jh} that $j(m)\leq k(\infty)$, and from
 Lemma~\ref{lem-tk} that $k(\infty)\leq t(q)$, and from
 Lemma~\ref{lem-kf} that $k(\infty)\leq f(n)$.  It follows that 
 \[ k(\infty) = k(\infty)\Smash k(\infty) \leq t(q)\Smash f(n)=t(q;n) \]
 as claimed.

 For the remaining statements, the case $m=\om$ follows easily from
 the cases $m\in\N$.  We will therefore assume that $m\in\Ni$.

 Suppose that $m\geq q$.  Then $T(m)$ is naturally a
 $T(q)$-module, so $J(m)=I(T(m))$ is naturally a $T(q)$-module, which
 implies (by Proposition~\ref{prop-bc-omni}(b)) that
 $t(q)\Smash j(m)=j(m)$.

 We now just need to show that when $m<q$ we have $T(q)\Smash J(m)=0$.
 By Proposition~\ref{prop-bc-omni}(d), this is equivalent to
 $F(T(q),T(m)/p)=0$.  If $q=\infty$ then this
 is~\cite[Lemma 3.2(b)]{ra:lrc}.  If $q<\infty$ then we can use
 Lemma~\ref{lem-jj-leq} to reduce to the case $m=q-1$, which
 is~\cite[Lemma 3.2(a)]{ra:lrc}.  
\end{proof}

\begin{lemma}\lbl{lem-kj}
 For all $n\in\Ni$ and $m\in\Nw$ we have $k(n)\Smash j(m)=0$.
\end{lemma}
\begin{proof}
 First suppose that $n<\infty$, so $k(n)\Smash k(\infty)=0$.  We have
 $j(m)\leq k(\infty)$ by Lemma~\ref{lem-jh}, so $k(n)\Smash j(m)=0$.

 Now consider the case where $n=\infty$ and $m\in\N$.  By
 Lemma~\ref{lem-tj} we have $t(m+1)\Smash j(m)=0$, but
 $k(\infty)\leq t(m+1)$ by Lemma~\ref{lem-tk}, so
 $k(\infty)\Smash j(m)=0$.  The case $m=\om$ follows from this.

 Finally, consider the case where $n=m=\infty$.  Here the claim is
 that $H/p\Smash IBP=0$, or equivalently that $F(H/p,BP/p)=0$.  This
 is the first step in the proof of~\cite[Theorem 2.2]{ra:lrc}.
\end{proof}

\begin{lemma}\lbl{lem-split}
 For all $i,j\in\N$ we have $k'(i)\Smash k'(i)=k'(i)$, but
 $k'(i)\Smash k'(j)=0$ for $i\neq j$.  We also have
 $k'(i)\Smash f(j)=0$ if $i<j$, and $k'(i)\Smash f(j)=k'(i)$ if
 $i\geq j$.  Finally, we have $f(n)=k'(n)\Wedge f(n+1)$.
\end{lemma}
\begin{proof}
 If $i<j$ then $v_i$ is nilpotent in $MU_*(F(i)\Smash F(j))$, so the
 Nilpotence Theorem tells us that $w_i\Smash 1_{F(j)}$ is nilpotent as
 a self-map of $F(i)\Smash F(j)$, so $K'(i)\Smash F(j)=0$, so
 $k'(i)\Smash f(j)=0$.  It is clear that $k'(j)\leq f(j)$, so we also
 have $k'(i)\Smash k'(j)=0$ when $i<j$.  By symmetry, this actually
 holds whenever $i\neq j$.  

 Next, Proposition~\ref{prop-bc-omni}(c) gives
 \[ f(n)=\bc{F(n)/v}\Wedge\bc{F(n)[v^{-1}]}=
          \bc{F(n)/v}\Wedge k'(n).
 \]
 The Thick Subcategory Theorem shows that $\bc{F(n)/v}=f(n+1)$, so
 $f(n)=f(n+1)\Wedge k'(n)$ (and we saw above that
 $f(n+1)\Smash k'(n)=0$).  An induction based on this shows that 
 $1=f(0)=f(j)\Wedge\bigWedge_{m<j}k'(m)$.  We can multiply this by $k'(i)$
 and use the relations that we have already established to get
 $k'(i)\Smash k'(i)=k'(i)$ if $i<j$, and $k'(i)\Smash f(j)=k'(i)$ for
 $i\geq j$.
\end{proof}

The next result is closely related to~\cite[Section 1]{ho:blf}.
\begin{lemma}\lbl{lem-kkt}
 For all $n\in\N$ we have
 $k(n)=t(\infty)\Smash k'(n)=t(\infty)\Smash k(n)$ and 
 $t(\infty)\Smash a(n)=0$.
\end{lemma}
\begin{proof}
 By construction, the spectrum $T(\infty)\Smash K'(n)$ is obtained by
 inverting the self-map $u=1_{BP}\Smash w_n$ of $BP\Smash F(n)$.
 However, we chose $w_n$ to be good, which means
 that $u$ is the same as $v_n\Smash 1_{F(n)}$, so
 $T(\infty)\Smash K'(n)=v_n^{-1}BP\Smash F(n)$.  We also know
 from~\cite[Theorem 2.1]{ra:lrc} that
 $\bc{v_n^{-1}BP}=\bc{E(n)}=\bigWedge_{i\leq n}k(i)$, and it follows
 that $t(\infty)\Smash k'(n)=k(n)$ as claimed.  This is the same as
 $t(\infty)\Smash k(n)$ by Lemma~\ref{lem-tk}.

 Now recall that $A(n)=K'(n)\Smash C_nS$, and by definition we have
 $E(n)\Smash C_nS=0$.  As $\bc{T(\infty)\Smash K'(n)}=\bc{E(n)}$, this
 gives 
 \[ T(\infty) \Smash A(n) = T(\infty)\Smash K'(n)\Smash C_nS = 0, \]
 so $t(\infty)\Smash a(n)=0$.
\end{proof}

\begin{lemma}\lbl{lem-split-k}
 For all $n\in\N$ we have $k(n)\leq k'(n)$ and
 $k'(n)\Smash k(n)=k(n)$, whereas $k'(n)\Smash k(m)=0$ for $m\neq n$.
\end{lemma}
\begin{proof}
 Multiply the equations in Lemma~\ref{lem-split} by $t(\infty)$ and
 then use Lemma~\ref{lem-kkt}.
\end{proof}

\begin{lemma}\lbl{lem-ff}
 For all $n,n'\in\N$ we have $f(n)\Wedge f(n')=f(\min(n,n'))$ and
 $f(n)\Smash f(n')=f(\max(n,n'))$.
\end{lemma}
\begin{proof}
 Recall that $F(n)$ has type $n$, which means that $K(i)_*F(n)$ is
 zero for $i<n$, and nonzero for $i\geq n$.  It follows that
 $F(n)\Wedge F(n')$ has type $\min(n,n')$, and $F(n)\Smash F(n')$ has
 type $\max(n,n')$.  The Thick Subcategory Theorem tells us that the
 Bousfield class of a finite $p$-local spectrum depends only on its
 type, so $f(n)\Wedge f(n')=f(\min(n,n'))$ and
 $f(n)\Smash f(n')=f(\max(n,n'))$.
\end{proof}

\begin{lemma}\lbl{lem-aa}
 The elements $a(n)$ satisfy $a(n)\Smash a(n)=a(n)\leq k'(n)$, and
 $a(n)\Smash a(m)=0$ for $m\neq n$.  Thus, the element
 $\ep(n)=\bigWedge_{i<n}a(i)$ is idempotent for all $n\in\Ni$.
\end{lemma}
This is also proved in~\cite[Section 1]{ho:blf}.
\begin{proof}
 First, the Smash Product Theorem~\cite[Theorem 7.5.6]{ra:nps} tells
 us that $A(n)=K'(n)\Smash C_nS$, so $a(n)\leq k'(n)$, so
 $a(n)\Smash a(m)=0$ for $m\neq n$ by Lemma~\ref{lem-split}.  We also
 have $C_nS\Smash C_nS=C_n(C_nS)=C_nS$ by the basic theory of
 Bousfield localisation, and in combination with Lemma~\ref{lem-split}
 this gives $a(n)\Smash a(n)=a(n)$.  
\end{proof}

\begin{corollary}\lbl{cor-ef}
 For all $n$ in $\N$ we have $\ep(n)\Smash f(n)=0$.
\end{corollary}
\begin{proof}
 As $\ep(n)=\bigWedge_{i<n}a(i)$, it will be enough to show that
 $a(i)\Smash f(n)=0$ for $i<n$.  The Lemma shows that
 $a(i)\leq k'(i)$, and $k'(i)\Smash f(n)=0$ by
 Lemma~\ref{lem-split}, so $a(i)\Smash f(n)=0$ as required.
\end{proof}

\begin{lemma}\lbl{lem-at}
 The elements $a(n)$ satisfy $a(n)\Smash k(m)=0$ for all $n$ and $m$,
 and $k'(n)=k(n)\Wedge a(n)$. 
\end{lemma}
This is also proved in~\cite[Section 1]{ho:blf}.
\begin{proof}
 We saw in Lemma~\ref{lem-kkt} that $t(\infty)\Smash a(n)=0$, and
 $k(m)\leq t(\infty)$ (even for $m=\infty$) by Lemma~\ref{lem-tk}, so
 $k(m)\Smash a(n)=0$.  Next, it follows from the Smash Product Theorem
 that $L_nS\Smash X=0$ iff $L_nX=0$ iff $E(n)\Smash X=0$, which means
 that $\bc{L_nS}=\bc{E(n)}$.  This is also the same as
 $\bigWedge_{i=0}^nk(i)$, by~\cite[Theorem 2.1]{ra:lrc}.  We can
 multiply by $k'(n)$ and use Lemma~\ref{lem-split-k} to get
 $\bc{L_nK'(n)}=k(n)$.  

 We also have a fibration 
 \[ A(n) = C_nS\Smash K'(n) \to K'(n) \to L_nS\Smash K'(n), \]
 which easily gives 
 \[ k'(n) = \bc{K'(n)} = \bc{A(n)}\Wedge\bc{L_nS\Smash K'(n)}
     = a(n) \Wedge k(n)
 \]
 as claimed.
\end{proof}

\begin{lemma}\lbl{lem-jj}
 For all $m,m'\in\Nw$ we have $j(m)\Smash j(m')=0$. 
\end{lemma}
\begin{proof}
 Lemma~\ref{lem-jh} gives $j(m)\leq k(\infty)$, and Lemma~\ref{lem-kj}
 gives $k(\infty)\Smash j(m')=0$, so $j(m)\Smash j(m')=0$. 
\end{proof}

\begin{lemma}\lbl{lem-jk-inf}
 If $m\in\Nw$ and $U\sse\N$ is infinite then $j(m)\leq k(U)$. 
\end{lemma}
\begin{proof}
 In view of Lemma~\ref{lem-jj-leq} we may assume that $m=\infty$, so
 $T(m)=BP$.  Suppose that $K(U)\Smash X=0$.  Hovey proved
 as~\cite[Corollary 3.5]{ho:blf} that $BP^\wedge_p$ is $K(U)$-local,
 so the spectrum $BP/p=(BP^\wedge_p)/p$ is also $K(U)$-local, so 
 $F(X,BP/p)=0$.  It follows by Proposition~\ref{prop-bc-omni}(d)
 that $J(\infty)\Smash X=0$.  We conclude that $j(\infty)\leq k(U)$,
 as claimed. 
\end{proof}

\begin{corollary}\lbl{cor-jk-big}
 If $m\in\Nw$, and $U\sse\Ni$ is big then $j(m)\leq k(U)$.
\end{corollary}
\begin{proof}
 This is just the conjunction of Lemmas~\ref{lem-jh}
 and~\ref{lem-jk-inf}. 
\end{proof}

\begin{lemma}\lbl{lem-fj}
 For all $n\in\N$ and $m\in\Nw$ we have $f(n)\Smash j(m)=j(m)$, and
 $a(n)\Smash j(m)=k'(n)\Smash j(m)=0$.
\end{lemma}
The statements with $m=0$ are contained in~\cite[Lemma 7.1]{hopa:sbl}.
\begin{proof}
 The claims for $m=\om$ follow easily from the claims for $m\in\N$, so
 we will assume that $m\in\Ni$.

 We first prove that $f(n)\Smash j(m)=j(m)$.  This is immediate when
 $n=0$, and follows from Proposition~\ref{prop-bc-omni}(d) when
 $n=1$, so we can suppose that $n>1$.  It is clear that
 $f(n)\Smash j(m)\leq j(m)$, so we just need the reverse inequality.
 Suppose that $X\Smash F(n)\Smash J(m)=0$, or equivalently
 $F(X\Smash F(n),T(m)/p)=0$.  We chose $F(n)$ to be self-dual, so 
 $F(X,F(n)\Smash T(m)/p)=0$.  By the Thick Subcategory Theorem, we can
 replace $F(n)$ here by any other finite spectrum of type $n$, so in
 particular $F(X,F(n-1)/w_{n-1}^k\Smash T(m)/p)=0$ for all $k$.  As
 $n>1$, a connectivity argument shows that $F(n-1)\Smash T(m)/p$ is
 the homotopy inverse limit of the spectra
 $F(n-1)/w_{n-1}^k\Smash T(m)/p$, so we see that 
 $F(X,F(n-1)\Smash T(m)/p)=0$.  By reversing the previous steps, we
 get $X\Smash F(n-1)\Smash J(m)=0$.  This gives
 $f(n)\Smash j(m)=f(n-1)\Smash j(m)$, which is the same as $j(m)$ by
 induction. 

 We can multiply the relation $f(n+1)\Smash j(m)=j(m)$ by $k'(n)$ and
 use $k'(n)\Smash f(n+1)=0$ (from Lemma~\ref{lem-split}) to get
 $k'(n)\Smash j(m)=0$.  We also have $a(n)\leq k'(n)$ by
 Lemma~\ref{lem-at}, so $a(n)\Smash j(m)=0$.
\end{proof}

\section{The main theorem}

By considering the phenomena in Lemma~\ref{lem-t-aux-ops} below, we
see that $\phi$ is unlikely to preserve either $\Wedge$ or $\Smash$
unless TC holds.  However, if we pass to $\bCL$ then we have the
following.

\begin{theorem}\lbl{thm-main}
 The map $\bphi=\pi\phi\:\CA\to\bCL$ is an injective homomorphism of
 ordered semirings. 
\end{theorem}
\begin{proof}
 See Corollary~\ref{cor-bCL-ops} and Corollary~\ref{cor-inj} below.
\end{proof}

We must show that the rules in Definition~\ref{defn-A} are valid as
equations in $\bCL$.  In fact, most of them are already valid in
$\CL$:
\begin{lemma}\lbl{lem-CL-ops}
 The rules for $t\Smash j$, $t\Smash k$, $j\Smash j$, $j\Smash k$,
 $k\Smash k$, $j\Wedge j$, $j\Wedge k$ and $k\Wedge k$ are all valid in
 $\CL$. 
\end{lemma}
(More concisely, these are all the rules where the right hand side
does not involve $t$.)
\begin{proof}\leavevmode
 \begin{itemize}
  \item Consider the element $x=t(q,T)\Smash j(m',S')$.  Let $n$ be
   minimal such that $[n,\infty]\sse T$.  Then $x$ is the wedge of
   terms $u_1=t(q;n)\Smash j(m')$ and 
   $u_2=t(q;n)\Smash k(S'))$ and $u_3=k(T)\Smash j(m')$
   and $u_4=k(T\cap S')$.  Corollary~\ref{cor-tk} tells us that
   $u_2=k(S'\cap [n,\infty])\leq u_4$.  We have $u_3=0$ by
   Lemma~\ref{lem-kj}, and $u_1=t(q)\Smash j(m')$ by
   Lemma~\ref{lem-fj}.  If $q\leq m'$ then Lemma~\ref{lem-tj} gives
   $u_1=j(m')$ and so $x=u_1\Wedge u_4=j(m',T\cap S')$.  If $q>m'$ then
   the same lemma gives $u_1=0$ and so $x=u_4=k(T\cap S')$.  
  \item Consider the element $x=t(q,T)\Smash k(U')$.  This is the
   wedge of the terms $u_1=t(q;n)\Smash k(U')=k(U'\cap[n,\infty])$ and 
   $u_2=k(T)\Smash k(U')=k(T\cap U')\geq u_1$, so $x=u_2=k(T\cap U')$ as
   required. 
  \item Consider the element $x=j(m,S)\Smash j(m',S')$.  This is the
   wedge of terms $u_1=j(m)\Smash j(m')$ and $u_2=j(m)\Smash k(S')$
   and $u_3=k(S)\Smash j(m')$ and $u_4=k(S\cap S')$.  The first three
   terms are zero by Lemmas~\ref{lem-jj} and~\ref{lem-kj}, so
   $x=u_4=k(S\cap S')$.
  \item Consider the element $x=j(m,S)\Smash k(U')$.  This is the
   wedge of terms $u_1=j(m)\Smash k(U')$ and $u_2=k(S\cap U')$.  We
   have $u_1=0$ by Lemma~\ref{lem-kj} so $x=u_2=k(S\cap U')$.
  \item We know from Lemma~\ref{lem-kk} that
   $k(U)\Smash k(U')=k(U\cap U')$.
  \item Put $x=j(m,S)\Wedge j(m',S')$.  Then
   $x=j(m)\Wedge j(m')\Wedge k(S\cup S')$, but $j(m)\Wedge j(m')=j(\max(m,m'))$ 
   by Lemma~\ref{lem-jj}, which gives $x=j(\max(m,m'),S\cup S')$.
  \item Put $x=j(m,S)\Wedge k(U')=j(m)\Wedge k(S\cup U')$.  If $U'$ is
   big then so is $S\cup U'$, so $j(m)\leq k(S\cup U')$ by
   Corollary~\ref{cor-jk-big}, so $x=k(S\cup U')$.  On the other hand,
   we are assuming that $S$ is small, so if $U'$ is small then
   $S\cup U'$ will also be small, so $j(m,S\cup U')$ is defined and is
   equal to $x$.
  \item We know from Lemma~\ref{lem-kk} that
   $k(U)\Wedge k(U')=k(U\cup U')$.
 \end{itemize}
\end{proof}

For the remaining rules, we have the following modified statement:
\begin{lemma}\lbl{lem-t-aux-ops}
 The following rules are valid in $\CL$ (provided that $n$ is large
 enough for the terms on the left to be defined):
 \begin{align*}
  t(q,T;n)\Smash t(q',T';n) &= t(\max(q,q'),T\cap T';n)) \\
  t(q,T;n)\Wedge t(q',T';n)   &= t(\min(q,q'),T\cup T';n) \\
  t(q,T;n)\Wedge j(m',S')     &= t(q,T\cup S';n) \\
  t(q,T;n)\Wedge k(U')        &= t(q,T\cup U';n).
 \end{align*}
\end{lemma}
\begin{proof}\leavevmode
 \begin{itemize}
  \item Consider the element $x=t(q,T;n)\Smash t(q',T';n)$.  This is
   the wedge of the terms $u_1=t(q;n)\Smash t(q';n')$ and
   $u_2=t(q;n)\Smash k(T')=k(T'\cap[n,\infty])$ and
   $u_3=k(T)\Smash t(q';n)=k(T\cap[n,\infty])$ and $u_4=k(T\cap T')$.
   We are assuming that $[n,\infty)\sse T$ and $[n,\infty)\sse T'$, so
   $u_2,u_3\leq u_4$.  We also have $u_1=t(\max(q,q');n)$ by
   Lemmas~\ref{lem-tt} and~\ref{lem-ff}.  This leaves
   $x=t(\max(q,q'),T\cap T';n)$.
  \item We have
   $t(q,T;n)\Wedge t(q',T';n)=t(q;n)\Wedge t(q';n)\Wedge k(T\cup T')$, and
   $t(q;n)\Wedge t(q';n)=t(\min(q,q');n)$ by Lemma~\ref{lem-tt}, which
   leaves $t(\min(q,q'),T\cup T';n)$.
  \item Put $x=t(q,T;n)\Wedge j(m',S')$.  Then
   $x=t(q;n)\Wedge k(T\cup S')\Wedge j(m')$, 
   but $j(m')\leq k(\infty)\leq t(q;n)$ by Lemma~\ref{lem-jh}, 
   so we can drop that term, giving $x=t(q,T\cup S';n)$.
  \item We have
   $t(q,T;n)\Wedge k(U')=t(q;n)\Wedge k(T)\Wedge k(U')=
    t(q;n)\Wedge k(T\cup U') = t(q,T\cup U';n)
   $.
 \end{itemize}
\end{proof}

\begin{lemma}\lbl{lem-t-aux}
 In $\bCL$ the element $t(q,T;n)$ is independent of the choice of
 $n$. 
\end{lemma}
\begin{proof}
 It is clear that $t(q,T;n)\geq t(q,T;n+1)$ in $\CL$, and it will
 suffice to show that this becomes an equality in $\bCL$.  We
 have 
 \[ f(n) = f(n+1)\Wedge k'(n) = f(n+1)\Wedge k(n) \Wedge a(n) \]
 by Lemmas~\ref{lem-split} and~\ref{lem-at}.  In conjunction with
 Lemma~\ref{lem-tk} this gives 
 \[ t(q;n) = t(q;n+1) \Wedge k(n)\Wedge (t(q)\Smash a(n)). \]
 However, we are assuming that $[n,\infty]\sse T$, so $n\in T$, so
 $k(n)\Wedge k(T)=k(T)$, so 
 \[ t(q,T;n) = t(q,T;n+1) \Wedge (t(q)\Smash a(n)). \]
 The extra term is less than or equal to $\ep(n+1)$, so it is killed
 by the homomorphism $\CL\to\CL/\ep(n)\to\bCL$.
\end{proof}

\begin{corollary}\lbl{cor-bCL-ops}
 All the relations in Definition~\ref{defn-A} are valid as equations
 in $\bCL$, so the map $\bphi=\pi\phi\:\CA\to\bCL$ is a homomorphism
 of semirings.
\end{corollary}
\begin{proof}
 This is clear from Lemmas~\ref{lem-CL-ops}, \ref{lem-t-aux-ops}
 and~\ref{lem-t-aux}.  
\end{proof}

\begin{remark}\lbl{rem-hCL}
 As well as $\bCL$, we can also consider the object
 $\hCL=\CL/\ep(\infty)$.  The canonical map $\CL\to\hCL$ then factors
 through $\bCL$, so we see that the composite $\CA\to\CL\to\hCL$ is
 also a homomorphism of ordered semirings.  This has the advantage
 that $\hCL$ is completely distributive, which we cannot prove for
 $\bCL$.  However, we do not know whether the map $\CA\to\hCL$ is
 injective. 
\end{remark}

\begin{definition}\lbl{defn-sg}
 We recall that $\CP$ denotes the set of subsets of $\Ni$, and we
 define maps $\sg_i\:\CL\to\CP$ by 
 \begin{align*}
  \sg_1(x) &= \{i\in\Ni \st k(i)\Smash x\neq 0\} 
            = \{i\in\Ni \st k(i)\Smash x=k(i)\}
            = \{i\in\Ni \st x\geq k(i)\} \\
  \sg_2(x) &= \{i\in\Ni \st j(i)\Smash x\neq 0\} \\
  \sg_3(x) &= \{i\in\Ni \st x \geq j(i)\}.
 \end{align*}
 (The three versions of $\sg_1$ agree by Lemma~\ref{lem-Kn-mod}.)
\end{definition}

\begin{lemma}\lbl{lem-bsg}
 There are maps $\bsg_r\:\bCL\to\CP$ (for $r=0,1,2$) with
 $\bsg_r\circ\pi=\sg_r\:\CL\to\CP$.  
\end{lemma}
\begin{proof}
 Lemmas~\ref{lem-at} and~\ref{lem-fj} show that
 $k(i)\Smash\ep(\infty)=j(i)\Smash\ep(\infty)=0$.  It follows that
 when $r\leq 2$ we have $\sg_r(\ep(n)\Wedge x)=\sg_r(x)$ for all
 $n\in\Ni$ and all $x\in\CL$.  This means that $\sg_1$ and $\sg_2$
 factor through $\bCL$ (or even $\hCL$) as claimed.

 Now consider $\sg_3$.  If $x\geq j(i)$, then of course
 $\ep(n)\Wedge x\geq j(i)$ for all $n\in\N$.  Conversely, suppose that
 $n\in\N$ and $\ep(n)\Wedge x\geq j(i)$.  It follows that 
 \[ f(n)\Smash (\ep(n)\Wedge x) \geq f(n)\Smash j(i). \]
 The right hand side is $j(i)$ by Lemma~\ref{lem-fj}.  On the left
 hand side, we have $f(n)\Smash\ep(n)=0$ by Corollary~\ref{cor-ef}.
 We therefore have $x\geq f(n)\Smash x\geq j(i)$.
 Putting this together, we see that $\sg_3(\ep(n)\Wedge x)=\sg_3(x)$
 for all $n\in\N$, so $\sg_3$ factors through $\bCL$.  (It is not
 clear, however, whether $\sg_3$ factors through $\hCL$.)
\end{proof}

\begin{corollary}\lbl{cor-inj}
 The map $\bphi\:\CA\to\bCL$ is injective.
\end{corollary}
\begin{proof}
 It is easy to check the following table of values of the maps $\sg_r$:
 \[ \begin{array}{|cl|c|c|c|} 
     \hline
            &               & \sg_1 & \sg_2      & \sg_3        \\ \hline
     t(q,T) &               & T     & [q,\infty] & [0,\infty]   \\ \hline
     j(m,S) &               & S     & \emptyset  & [0,m]\cap\Ni \\ \hline
     k(U)   &\text{(small)} & U     & \emptyset  & \emptyset    \\ \hline
     k(U)   &\text{(big)}   & U     & \emptyset  & [0,\infty]   \\ \hline
    \end{array}
 \]
 (In particular, we have $\sg_3(j(\om,S))=\N$ but $\sg_3(j(\infty,S))=\Ni$.)
 Now consider an element $x\in\CA$ with $\sg\phi(x)=u\in\CP^3$.  We
 see that: 
 \begin{itemize}
  \item If $u_2\neq\emptyset$ then $x=t(\min(u_2),u_1)$.
  \item If $u_2=\emptyset$ and $u_1$ is small and $u_3\neq\emptyset$
   then $x=j(\sup(u_3),u_1)$ (where the supremum is taken in $\Nw$).
  \item If $u_2=\emptyset$ and $u_1$ is small and $u_3=\emptyset$
   then $x=k(u_1)$.
  \item If $u_2=\emptyset$ and $u_1$ is big then $x=k(u_1)$.
 \end{itemize}
 This means that the composite $\sg\phi$ is injective, but this is the
 same as $\bsg\bphi$, so $\bphi$ is injective.
\end{proof}

We do not know whether $\bCL$ is complete.  However, we do have the
following partial result.

\begin{proposition}\lbl{prop-phi-join}
 Let $S$ be any subset of $\CA$, and put $a=\bigWedge S\in\CA$.  Then
 $\bphi(a)$ is the least upper bound for $\bphi(S)$ in $\bCL$.
\end{proposition}
\begin{proof}
 Let $V$ denote the set of upper bounds for $\bphi(S)$.  As $\bphi$ is
 a homomorphism of ordered semirings, it is clear that
 $\bphi(a)\in V$.  We must show that $\bphi(a)$ is the smallest
 element of $V$.

 Now let $\CI$ denote the ideal in $\CA$ generated by $S$, so $a$ is
 also the least upper bound for $\CI$.  This means that $a=\tht(\CI)$,
 where $\tht$ is as in Definition~\ref{defn-tht}.

 Note that the set
 \[ S' = \{x\in\CA\st\bphi(x)\leq v \text{ for all } v\in V\} \]
 is an ideal containing $S$, so it contains $\CI$.  This means that
 $V$ is also the set of upper bounds for $\bphi(\CI)$.

 In cases~(c) and~(d) of Definition~\ref{defn-tht},
 Lemma~\ref{lem-tht-max} tells us that $a\in\CI$, and the claim
 follows immediately from this.  We therefore need only consider
 cases~(a) and~(b), in which $\CI\sse\CA_j\cup\CA_k$.  

 Now let $v$ be an element of $\CL$ such that the image
 $\pi(v)\in\bCL$ lies in $V$.  This means that for all $x\in\CI$ there
 exists $n\in\N$ such that $\ep(n)\Wedge v\geq\phi(x)$ in $\CL$.  We
 must show that there exists $m$ such that $\ep(m)\Wedge v\geq\phi(a)$.

 Recall that
 \[ A = \bigcup_{x\in\CI}\tail(x) \sse \Ni, \]
 and Lemma~\ref{lem-tht-k} tells us that $k(i)\in\CI$ for all
 $i\in A$.  This means that $\ep(n)\Wedge v\geq k(i)$ for some
 $n\in\N$.  Lemma~\ref{lem-at} tells us that $\ep(n)\Smash k(i)=0$, so
 we can multiply the above relation by $k(i)$ to get
 \[ v \geq k(i)\Smash v \geq k(i)\Smash k(i) = k(i) \]
 in $\CL$.  This holds for all $i\in A$, and the element $k(A)\in\CL$
 is by definition the least upper bound in $\CL$ of the elements
 $k(i)$ with $i\in A$, so we get $v\geq k(A)$ in $\CL$.  In case~(a),
 this is already the desired conclusion.

 Finally, we consider case~(b), in which $A$ is small but the set 
 \[ M = \{m\in\Nw\st j(m)\in\head(\CI)\} \]
 is an infinite subset of $\N$.  Lemma~\ref{lem-tht-b} tells us that
 $j(m,A)\in\CI$ for all $m\in M$.  Thus, for each $m\in M$ there
 exists $n\in\N$ such that $\ep(n)\Wedge v\geq j(m,A)\geq j(m)$.  We
 now multiply this relation by $f(n)$.  Corollary~\ref{cor-ef} tells
 us that $f(n)\Smash \ep(n)=0$, and Lemma~\ref{lem-fj} gives
 $f(n)\Smash j(m)=j(m)$, so we have 
 \[ v \geq f(n)\Smash v \geq j(m). \]
 Now recall that $j(i)\leq j(i+1)$ for $i\in\N$, and that the element
 $j(\om)\in\CL$ is by definition the join in $\CL$ of these elements
 $j(i)$.  We therefore have $v\geq j(\om)$ in $\CL$, and we have
 already seen that $v\geq k(A)$,
 so $v\geq j(\om)\Wedge k(A)=j(\om,A)=a$ as required.
\end{proof}

\section{Index of popular Bousfield classes}

We next give a list of spectra $X$ together with corresponding
elements $x\in\CA$.  We write $X=x$ to indicate that $\bc{X}=\phi(x)$
in $\CL$, or $X\simeq x$ to indicate that $\pi\bc{X}=\pi\phi(x)$ in
$\bCL$.  As usual, everything is implicitly $p$-localised.

\begin{align}
 0 &= k(\emptyset)                                            \label{eq-O}   \\
 S = S^\wedge_p = T(0) &= t(0,\Ni)                            \label{eq-S}   \\
 S/p = S/p^\infty &= t(0,[1,\infty])                          \label{eq-Sp}  \\
 F(n) &= t(0,[n,\infty])                                      \label{eq-Fn}  \\ 
 H\Q = S\Q = I(H\Q) &= k(\{0\})                               \label{eq-HQ}  \\
 H/p = H/p^\infty = I(H) = I(H/p) = I(BP\ip{n}) &= k(\{\infty\}) \label{eq-Hp} \\
 H &= k(\{0,\infty\})                                         \label{eq-H}   \\
 v_n^{-1}F(n) = K'(n) &\simeq k(\{n\})                        \label{eq-Kn}  \\
 T(q) &= t(q,\Ni)                                             \label{eq-Tq}  \\
 BP = BP^\wedge_p = T(\infty) &= t(\infty,\Ni)                \label{eq-BP}  \\
 P(n) = BP/I_n &= t(\infty,[n,\infty])                        \label{eq-Pn}  \\
 B(n) = v_n^{-1}P(n) = K(n) = M_nS &= k(\{n\})                \label{eq-Bn}  \\
 IB(n) =  IK(n) &= k(\{n\})                                   \label{eq-IKn} \\
 E(n) = v_n^{-1}BP\ip{n} = v_n^{-1}BP = L_nS &= k([0,n])      \label{eq-En}  \\
 \widehat{E(n)} = L_{K(n)}S &= k([0,n])                       \label{eq-Enh} \\
 C_nS &\simeq t(0,[n+1,\infty])                               \label{eq-CnS} \\
 BP\ip{n} &= k([0,n]\cup\{\infty\})                           \label{eq-BPn} \\
 BP\ip{n}/I_n &= k(\{n,\infty\})                              \label{eq-kn}  \\
 KU = KO &= k(\{0,1\})                                        \label{eq-KU}  \\
 kU = kO &= k(\{0,1,\infty\})                                 \label{eq-kU}  \\
 Ell = TMF &= k(\{0,1,2\})                                    \label{eq-Ell} \\
 I(S) = I(T(0)) = I(F(n)) &= j(0,\emptyset)                   \label{eq-I}   \\
 I(S^\wedge_p) = I(S/p^\infty) &= j(0,\{0\})                  \label{eq-IQ}  \\
 I(T(m)) = I(T(m)\Smash F(n)) &= j(m,\emptyset)               \label{eq-J}
\end{align}

\begin{proof}\leavevmode
 \begin{itemize}
  \item[(\ref{eq-O})] Clear from the definitions. 
  \item[(\ref{eq-S})] Clear from the definitions together with
   Proposition~\ref{prop-bc-omni}(e).
  \item[(\ref{eq-Fn})] Clear from the definitions.
  \item[(\ref{eq-Sp})] From~(\ref{eq-Fn}) we see that 
   $t(0,[1,\infty])$ is the same as 
   \[ F(1)=F(U(1),U(1))=F(S/p,S/p)= D(S/p)\Smash S/p, \]
   and it is easy to check that this has the same Bousfield class as
   $S/p$.  Now $S/p^\infty$ can be described as the homotopy colimit
   of the spectra $S/p^n$, or as the cofibre of the map
   $S\to S[p^{-1}]$.  From the first description (together with the
   cofibrations $S/p^n\to S/p^{n+1}\to S/p$) we see that
   $\bc{S/p^\infty}\leq\bc{S/p}$.  The second description shows that
   $S/p^\infty\Smash \Sg^{-1}S/p\simeq S/p$, which gives
   $\bc{S/p}\leq\bc{S/p^\infty}$, so we have
   $\bc{S/p^\infty}=\bc{S/p}$.
  \item[(\ref{eq-HQ})] By definition we have $H\Q=k(\{0\})$, and it
   is standard that this is the same as $S\Q$.  Moreover, $I(H\Q)$ is
   a module over $H\Q$, so Proposition~\ref{prop-bc-omni}(b) tells us
   that the Bousfield class is the same as $H\Q$ provided that
   $I(H\Q)\neq 0$.  By definition we have
   $\pi_0(I(H\Q))=\Hom(\Q,\QZpl)$, which is nontrivial as required.
  \item[(\ref{eq-Hp})] By definition we have $H/p=k(\{\infty\})$, and
   this is the same as $H/p^\infty$ as a consequence
   of~(\ref{eq-Sp}).  Note that $\pi_*(I(H))=\Hom(\pi_{-*}(H),\Z/p^\infty)$,
   which is a copy of $\Z/p^\infty$ concentrated in degree zero, so
   $I(H)=H/p^\infty$.  A similar argument gives $I(H/p)=H/p$.  We next
   consider the classes $u(n)=\bc{I(BP\ip{n})}$.  These start with
   $u(0)=\bc{IH}=k(\infty)$, so it will suffice to prove that
   $u(n)=u(n-1)$ when $n>0$.  Proposition~\ref{prop-bc-omni}(c) gives
   \[ u(n)=\bc{I(BP\ip{n})/v_n} \Wedge\bc{I(BP\ip{n})[v_n^{-1}]}, \]
   and the first term is the same up to suspension as $u(n-1)$.  The
   second term is the colimit of the spectra $\Sg^{-k|v_n|}IBP\ip{n}$,
   which is trivial because the homotopy of $IBP\ip{n}$ is
   concentrated in nonpositive degrees.  The claim follows.
  \item[(\ref{eq-H})] Proposition~\ref{prop-bc-omni}(e) gives
   $\bc{H}=\bc{H/p}\Wedge\bc{H\Q}$, and this is $k(\{0,\infty\})$ by
   definition. 
  \item[(\ref{eq-Kn})] We have $K'(n)=v_n^{-1}F(n)$ by definition.  If
   TC fails, then this may be different from $k(\{n\})$ in $\CL$.
   However, Lemma~\ref{lem-at} tells us that $k'(n)=a(n)\Wedge k(n)$,
   so $K'(n)$ and $k(n)$ have the same image in $\bCL$, as we indicate
   by writing $K'(n)\simeq k(n)$. 
  \item[(\ref{eq-Tq})] True by definition.
  \item[(\ref{eq-BP})] Clear from the definitions together with
   Proposition~\ref{prop-bc-omni}(e).
  \item[(\ref{eq-Pn})] We have $P(n)=BP/I_n$ by definition.  Now
   consider a more general spectrum of the form $BP/J$, where $J$ is
   generated by an invariant regular sequence of length $n$.  Ravenel
   proved as~\cite[Theorem 2.1(g)]{ra:lrc} that
   $\bc{BP/J}=\bc{P(n)}$.  Using the theory of generalised Moore
   spectra~\cite[Chapter 6]{ra:nps}\cite[Chapter 4]{host:mkl} 
   we see that for suitable $J$ there is a finite spectrum $S/J$ of
   type $n$ such that $BP/J=BP\Smash S/J$.  By the Thick Subcategory
   Theorem we have $\bc{S/J}=f(n)$ and so
   $\bc{P(n)}=\bc{BP}\Smash f(n)=t(\infty,[n,\infty])$.
  \item[(\ref{eq-Bn})] We have $K(n)=k(\{n\})$ by definition, and the
   spectrum $B(n)=v_n^{-1}P(n)$ has the same Bousfield class
   by~\cite[Theorem 2.1(a)]{ra:lrc}.  We will discuss $M_nS$
   under~(\ref{eq-En}). 
  \item[(\ref{eq-IKn})] We first note that $IK(n)$ is a $K(n)$-module,
   and all $K(n)$-modules are free, and the homotopy groups of $IK(n)$
   have the same order as those of $K(n)$, so $IK(n)\simeq K(n)$ as
   spectra, so certainly $\bc{IK(n)}=k(\{n\})$.  Next, note that
   $IB(n)$ is a $B(n)$-module, so
   \[ \bc{IB(n)} = \bc{B(n)}\Smash \bc{IB(n)} = 
       \bc{K(n)} \Smash \bc{IB(n)},
   \]
   and this is either zero or $k(\{n\})$ by Lemma~\ref{lem-Kn-mod}.
   It cannot be zero because 
   \[ \pi_0(IB(n))=\Hom(\pi_0B(n),\QZpl)\neq 0, \] 
   so it must be $k(\{n\})$ as claimed.
  \item[(\ref{eq-En})] We have $E(n)=v_n^{-1}BP\ip{n}$ by definition.
   This has the same Bousfield class as $v_n^{-1}BP$ and as
   $K(\{0,\dotsc,n\})$ by parts~(b) and~(d)
   of~\cite[Theorem 2.1]{ra:lrc}.  Note that $E(n)\Smash X=0$ iff
   $L_nX=0$, which is equivalent to $L_nS\Smash X=0$
   by~\cite[Theorem 7.5.6]{ra:nps}.  This means that $L_nS$ also has
   the same Bousfield class.  Finally, recall that
   $M_nS=C_{n-1}L_nS=C_{n-1}S\Smash L_nS$.  This gives
   \[ \bc{M_nS}=\bc{\bigWedge_{i\leq n}K(i)\Smash C_{n-1}S}
               =\bc{\bigWedge_{i\leq n}C_{n-1}(K(i))}. 
   \]
   Here $K(i)$ is $E(n-1)$-local for $i<n$, and $E(n-1)$-acyclic for
   $i=n$, which gives $\bc{M_nS}=k(n)$ as claimed in~(\ref{eq-Bn}).
  \item[(\ref{eq-Enh})] This is part
   of~\cite[Proposition 5.3]{host:mkl} (where $\widehat{E(n)}$ is
   denoted by $E$, and $L_{K(n)}S$ is denoted by $\widehat{L}S$).  
  \item[(\ref{eq-CnS})] We know from~(\ref{eq-Fn}) that
   $t(0,[n+1,\infty])=f(n+1)$, and we must show that this becomes the
   same as $C_nS$ in $\bCL$.  Put
   \begin{align*}
    e  &= \bc{E(n)} = \bc{L_nS} = \bigWedge_{i\leq n}k(i) &
    e' &= \bigWedge_{i\leq n} k'(i) \\
    f  &= f(n+1) &
    f' &= \bc{C_nS}. 
   \end{align*}
   An induction based on Lemma~\ref{lem-split} gives
   $1=f(0)=e'\Wedge f$, with $e'\Smash f=0$.  Next, $C_nS$
   is by definition the fibre of the localisation map $S\to L_nS$, and
   this fibre sequence gives $1=e\Wedge f'$.  Moreover, $C_nS$
   and $F(n+1)$ are both $E(n)$-acyclic by construction, so
   $e\Smash f=e\Smash f'=0$ and $f'\Smash f=f$.  We now have
   \begin{align*}
    f  &= f\Smash 1 = f\Smash (e\Wedge f') = f\Smash f' \\
    f' &= f'\Smash 1 = f'\Smash (e'\Wedge f) = 
          (f'\Smash e') \Wedge (f'\Smash f) \\
       &= (f'\Smash e') \Wedge f.
   \end{align*}
   All of this is valid in $\CL$.  If we pass to $\bCL$ then $e$ and
   $e'$ become the same by~(\ref{eq-Kn}), so
   $f'\Smash e'=f'\Smash e=0$, so $f=f'$ as required.
  \item[(\ref{eq-BPn})] This is~\cite[Theorem 2.1(e)]{ra:lrc}.
  \item[(\ref{eq-kn})] By the same argument as for~(\ref{eq-Pn}), we
   have $\bc{BP\ip{n}/I_n}=\bc{BP\ip{n}}\Smash f(n)$.
   Using~(\ref{eq-BPn}) and Lemma~\ref{lem-kf}, this reduces to
   $k(\{n,\infty\})$.  
  \item[(\ref{eq-KU})] The spectrum $KU$ is Landweber exact with
   strict height one, so it is Bousfield equivalent to $E(1)$
   by~\cite[Corollary~1.12]{ho:blf}.  It is a theorem of
   Wood~\cite{wo:bab} that $KU=KO/\eta$, where $\eta\in\pi_1(KO)$ is
   the Hopf map.  This is essentially equivalent
   to~\cite[Proposition 3.2]{at:ktr}, and the same paper proves the
   standard fact that $\eta^3=0$ in $\pi_*(KO)$.  As $\eta$ is
   nilpotent we find that $KU$ generates the same thick subcategory as
   $KO$, and thus has the same Bousfield class.
  \item[(\ref{eq-kU})] We can take connective covers in the theorem of
   Wood to see that $kU=kO/\eta$, so $\bc{kU}=\bc{kO}$.  If $v$
   denotes the Bott element in $\pi_2(kU)$ then we have $kU/v=H$ and
   $kU[v^{-1}]=KU$, so $\bc{kU}=\bc{KU}\Wedge\bc{H}$, which is
   $k(\{0,1,\infty\})$ by~(\ref{eq-H}) and~(\ref{eq-KU}).
  \item[(\ref{eq-Ell})] Here $Ell$ is intended to denote any of the
   standard Landweber exact versions of elliptic cohomology, which all
   have strict height two, so $Ell=k(\{0,1,2\})$ 
   by~\cite[Corollary~1.12]{ho:blf}.  At primes $p>2$ the spectrum
   $TMF$ is itself a version of $Ell$ and so $\bc{TMF}=k(\{0,1,2\})$.
   For $p\in\{2,3\}$ it is known~\cite{homa:ech} that there is a
   finite spectrum $X$ of type $0$ such that $TMF\Smash X=E(2)$, so we
   again have the same Bousfield class.
  \item[(\ref{eq-I})] We have $I(S)=I(T(0))=j(0,\emptyset)$ by
   definition.  For any finite spectrum $X$, it is easy to see that
   $I(X)=DX\Smash I(S)$.  As $F(n)$ is self-dual we have
   $I(F(n))=F(n)\Smash I(S)$, and this has Bousfield class
   $j(0,\emptyset)$ by Lemma~\ref{lem-fj}.
  \item[(\ref{eq-IQ})] First, we have 
   \[ I(S^\wedge_p)/p = \Sg I((S^\wedge_p)/p) = \Sg I(S/p), \]
   which gives
   \[ \bc{I(S^\wedge_p)}\geq\bc{I(S/p)}=\bc{I(S)}=
          j(0,\emptyset).
   \]
   Next, there is a natural map $i\:S\to S^\wedge_p$,
   with cofibre $X$ say.  We find that $\pi_k(X)=0$ for $k\neq 0$, but
   that $\pi_0(X)=\Z_p/\Zpl$, which is a nontrivial rational vector
   space.  This gives a fibration $IX\to I(S^\wedge_p)\to IS$, giving 
   \[ \bc{I(S^\wedge_p)} \leq \bc{IX}\Wedge \bc{IS}. \]
   Here $IS$ is torsion and $IX$ is rational and nontrivial, so it
   follows that $I(S^\wedge_p)$ is not torsion, and so
   $\bc{I(S^\wedge_p)}\geq\bc{HQ}=k(\{0\})$.  Putting this together,
   we get $\bc{I(S^\wedge_p)}=\bc{IS}\Wedge\bc{H\Q}=j(0,\{0\})$ as
   claimed.  A similar proof works for $S/p^\infty$, using the
   defining cofibration $S\to S\Q\to S/p^\infty$.  
  \item[(\ref{eq-J})] We have $I(T(m))=j(m,\emptyset)$ by definition,
   and this is the same as $I(T(m)\Smash F(n))$ by Lemma~\ref{lem-fj}
   and the self-duality of $F(n)$.
 \end{itemize}
\end{proof}

\bibliographystyle{gtpart}

\end{document}